\documentclass[10pt,a4paper]{article}
\usepackage[a4paper]{geometry}
\usepackage{amssymb,latexsym,amsmath,amsfonts,amsthm}
\usepackage{graphicx}

\usepackage{epsfig}
\usepackage{comment}
\usepackage{subfigure}
\usepackage{graphicx,ebezier}
\usepackage{overpic}

\DeclareMathOperator{\diag}{diag} 

\newcommand{\er}{\mathbb{R}}
\newcommand{\cee}{\mathbb{C}}
\newcommand{\enn}{\mathbb{N}}
\newcommand{\zet}{\mathbb{Z}}
\newcommand{\lam}{\lambda}
\newcommand{\Lam}{\Lambda}

\newcommand{\bol}{\hfill\square\\}
\newcommand{\til}{\tilde}
\newcommand{\wtil}{\widetilde}

\renewcommand{\sp}{\mathrm{sp}\ }
\newcommand{\res}{\mathrm{Res}}
\newcommand{\mult}{\mathrm{mult}}
\newcommand{\cls}{\mathrm{cls}}

\newtheorem{theorem}{Theorem}[section]
\newtheorem{lemma}[theorem]{Lemma}
\newtheorem{proposition}[theorem]{Proposition}
\newtheorem{example}[theorem]{Example}
\newtheorem{corollary}[theorem]{Corollary}

\theoremstyle{definition}
\newtheorem{definition}[theorem]{Definition}

\theoremstyle{remark}

\newtheorem{remark}[theorem]{Remark}

\numberwithin{equation}{section}

\title{Equilibrium problem for the eigenvalues of banded
block Toeplitz matrices}

\date{}

\author{Steven Delvaux\footnotemark[1]}

\begin{document}

\renewcommand{\thefootnote}{\fnsymbol{footnote}}
\footnotetext[1]{Department of Mathematics, University of Leuven (KU Leuven),
Celestijnenlaan 200B, B-3001 Leuven, Belgium. email:
steven.delvaux\symbol{'100}wis.kuleuven.be. The author is a Postdoctoral Fellow
of the Fund for Scientific Research - Flanders (Belgium).}

\maketitle

\begin{abstract}
We consider banded block Toeplitz matrices $T_n$ with $n$ block rows and
columns. We show that under certain technical assumptions, the normalized
eigenvalue counting measure of $T_n$ for $n\to\infty$ weakly converges to one
component of the unique vector of measures that minimizes a certain energy
functional. In this way we generalize a recent result of Duits and Kuijlaars
for the scalar case. Along the way we also obtain an equilibrium problem
associated to an arbitrary algebraic curve, not necessarily related to a block
Toeplitz matrix.

For banded block Toeplitz matrices, there are several new phenomena that do not
occur in the scalar case: (i) The total masses of the equilibrium measures do
not necessarily form a simple arithmetic series but in general are obtained
through a combinatorial rule; (ii) The limiting eigenvalue distribution may
contain point masses, and there may be attracting point sources in the
equilibrium problem; (iii) More seriously, there are examples where the
connection between the limiting eigenvalue distribution of $T_n$ and the
solution to the equilibrium problem breaks down. We provide sufficient
conditions guaranteeing that no such breakdown occurs; in particular we show
this if $T_n$ is a Hessenberg matrix.
\end{abstract}
\maketitle                   






\setcounter{tocdepth}{1} \tableofcontents

\section{Introduction} \label{section:settingblock}
\label{section:intro}

Let $r\in\enn:=\{1,2,3,\ldots\}$ and let there be given a set of $r\times r$
matrices
$$A_k\in\cee^{r\times r},\qquad k=-\alpha,\ldots,\beta,$$ for some
$\alpha,\beta\in\enn$. These matrices are encoded by the matrix-valued Laurent
polynomial (also called \emph{symbol})
\begin{equation}\label{def:A} A(z) = A_{-\alpha}z^{-\alpha}+\ldots+A_{\beta}
z^{\beta}.
\end{equation}
For $n\in\enn$ define the \emph{block Toeplitz matrix} $T_n(A)$ associated to
the symbol $A(z)$ by
\begin{equation}T_n(A) =
\begin{pmatrix} A_{i-j}
\end{pmatrix}_{i,j=1}^{n}\in\cee^{rn\times rn},
\end{equation}
where we put $A_k\equiv 0$ if $k>\beta$ or $k<-\alpha$. Explicitly,
\begin{equation}
T_n(A) =
\begin{pmatrix} \label{blockToeplitzmatrix}
A_0 & \ldots & A_{-\alpha}  & & & 0 \\
\vdots & \ddots & & \ddots & &  \\
A_{\beta} &  & \ddots & & \ddots &  \\
 & \ddots & & \ddots &  & A_{-\alpha} \\
  & & \ddots &  & \ddots & \vdots \\
0   & & & A_{\beta} & \ldots & A_{0}
\end{pmatrix}_{rn\times rn}.
\end{equation}

In this paper we are interested in the limiting behavior of the eigenvalues of
$T_n(A)$ for $n\to\infty$. It is known that under certain technical assumptions
\cite{Widom1}, the eigenvalue counting measure has a weak limit supported on a
certain curve $\Gamma_0$ in the complex plane. An example of this phenomenon is
shown in Figure~\ref{fig:illint} for the symbol
\begin{equation}\label{example:sym:int}
A(z)=\begin{pmatrix}z^2&1\\
z^{-1}+z& 0\end{pmatrix};\end{equation} see B\"ottcher-Grudsky \cite{BG} for
many more illustrations of this type.

\begin{figure}[t]
\label{fig:illint}
\begin{center}
        \includegraphics[scale=0.35,angle=270]{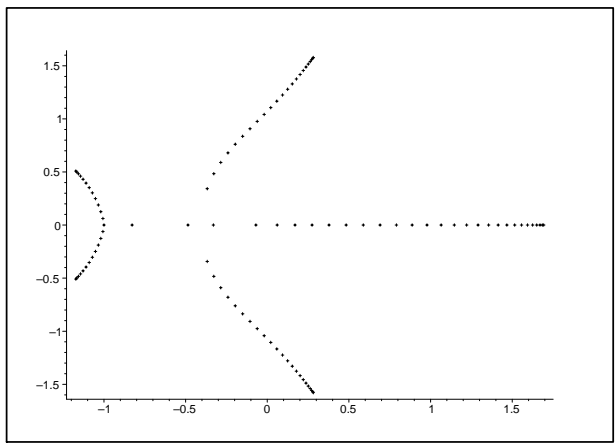}
\end{center}
\caption{Point plot in the complex plane of the eigenvalues of the banded block
Toeplitz matrix $T_n(A)$, $n=50$, with $r=2$ and symbol
\eqref{example:sym:int}, computed in Maple with 60 digit precision. For
$n\to\infty$ the eigenvalues accumulate on a curve $\Gamma_0\subset\cee$ which
consists of six analytic arcs connecting the points $-1$, $-0.42$, $-1.17\pm
0.51 i$, $0.28\pm 1.58 i$ and $1.70$ (using two digits of precision). Three
arcs are emanating with equal angles from $-1$ and four arcs from $-0.42$. The
limiting eigenvalue distribution of $T_n(A)$ for $n\to\infty$ exists as an
absolutely continuous measure $\mu_0$ on~$\Gamma_0$.}
\end{figure}

In the case of scalar banded Toeplitz matrices, $r=1$, it was recently shown by
Duits-Kuijlaars~\cite{DK} that the limiting eigenvalue distribution of $T_n(A)$
satisfies a (vector) equilibrium problem that is constructed out of the symbol.
The goal of this paper is to extend this result to the block
case~$r>1$.\smallskip

Let us first review some known results in the literature, following the
exposition in \cite{DD,DK}. We denote the eigenvalue spectrum of $T_n(A)$ by
$$ \sp T_n(A) = \{\lam\in\cee\mid \det(T_n(A)-\lam I_{rn})=0\},
$$
where in general we use $I_k$ to denote the identity matrix of size $k$ by $k$.
Following Schmidt-Spitzer~\cite{SchmidtSpitzer}, we define two limiting sets of
the spectrum: we define
$$ \liminf_{n\to\infty} \sp T_n(A)
$$
to be the set of all $\lam\in\cee$ for which there exists a sequence
$(\lam_n)_{n\in\enn}$, with $\lam_n\in\sp T_n(A)$ converging to $\lam$.
Similarly we define
$$ \limsup_{n\to\infty} \sp T_n(A)
$$
to be the set of all $\lam\in\cee$ for which there exists a sequence
$(\lam_n)_{n\in\enn}$, with $\lam_n\in\sp T_n(A)$ having a subsequence
converging to $\lam$.

Under certain assumptions \cite{Widom1}, the above limiting sets can be
described in terms of the algebraic equation
\begin{equation}\label{def:f} 0=f(z,\lam) := \det (A(z)-\lam I_r).
\end{equation}
Note that each entry of the matrix $A(z)-\lam I_r$ is a Laurent polynomial in
$z$, by virtue of \eqref{def:A}. Hence $f(z,\lam)=\det (A(z)-\lam I_r)$  is a
Laurent polynomial in $z$ as well, and we can write it in the form
\begin{equation}\label{f:expansion}
f(z,\lam) = \sum_{k=-q}^{p} f_k(\lam) z^k,
\end{equation}
for certain $p,q\in\enn\cup\{0\}$. The coefficients $f_k(\lam)$ are polynomials
in $\lam$ of degree at most $r$. More precisely,
\begin{equation}\label{f:expansionlargelambda}
\deg f_k=\left\{\begin{array}{ll}
 r, & \textrm{if }k=0, \\
\leq r-1, & \textrm{if }k\neq 0,\end{array}\right.
\end{equation}
on account of \eqref{def:f}--\eqref{f:expansion}. We assume that the numbers
$p,q$ in \eqref{f:expansion} are such that the outermost coefficients
$f_{-q}(\lam)$ and $f_p(\lam)$ are not identically zero as a function of
$\lam$. To avoid trivial cases we will always assume that
\begin{equation}\label{assumption:pq}
\min(p,q)\geq 1.
\end{equation} This is justified since if $\min(p,q)=0$, then
$\det (T_n(A)-\lam I_{rn})=C_0(\lam)f_0(\lam)^{n+\alpha}$ for a certain
rational function $C_0(\lam)$, by Proposition~\ref{proposition:Widomformula}
below. In that case the eigenvalues of $T_n(A)$ are trivially obtained.
\smallskip

For any $\lam\in\cee$ with $f_{p}(\lam)\neq 0$, we consider
\begin{equation}\label{Laurentpol:shifted} z^q f(z,\lam) = \sum_{k=-q}^{p}
f_{k}(\lam) z^{k+q}
\end{equation}
as a polynomial in $z$ of degree $p+q$. We order its roots $z=z(\lam)$
(counting multiplicities) by absolute value as
\begin{equation}\label{ordering:roots} 0\leq |z_1(\lam)|\leq |z_2(\lam)|\leq\ldots\leq
|z_{p+q}(\lam)|.
\end{equation}
If $\lam$ is such that two or more subsequent roots in \eqref{ordering:roots}
have the same absolute value, then we may arbitrarily label them so that
\eqref{ordering:roots} is satisfied. For the special values of $\lam$ for which
$f_{p}(\lam)= 0$, the polynomial \eqref{Laurentpol:shifted} has degree less
than $p+q$, say $p+q-j$, and in that case we order its roots
$z_1(\lambda),\ldots,z_{p+q-j}(\lambda)$ as in \eqref{ordering:roots} and set
$z_{p+q-j+1}(\lambda)=\ldots=z_{p+q}(\lambda)=\infty$, compare with \cite{DD}.
We also use the latter convention if $\lam\in\cee$ is such that
$f(z,\lam)\equiv 0$. Thus in that case we put $z_j(\lam)=\infty$ for
all~$j=1,\ldots,p+q$.

Each of the roots $z_j(\lam)$ is finite and non-zero, except when $\lam$
belongs to the set
\begin{equation}\label{def:laml}
\Lambda := \{ \lam\in\cee\mid f_{-q}(\lam)
f_{p}(\lam)=0\}.
\end{equation}
By virtue of \eqref{f:expansionlargelambda}, the set $\Lam$ has cardinality
$|\Lam|\leq 2r-2$. In particular $\Lam$ is empty in the scalar case $r=1$.
\smallskip

Define the set
\begin{equation}\label{def:curves0} \Gamma_0 := \{\lam\in\cee\mid |z_{q}(\lam)| = |z_{q+1}(\lam)|
\}.
\end{equation}
For the case of scalar banded Toeplitz matrices $r=1$ it is known that
$\Gamma_0$ is a curve consisting of a finite number of analytic arcs and having
no isolated points, and moreover the eigenvalues of $T_n(A)$ accumulate on
$\Gamma_0$ in the sense that
\begin{equation}\label{ss:scalar} \liminf_{n\to\infty} \sp T_n(A) = \limsup_{n\to\infty} \sp T_n(A) = \Gamma_0.
\end{equation}
These results were shown by Schmidt and Spitzer \cite{SchmidtSpitzer}. The same
authors also showed that the limiting eigenvalue distribution $\mu_0$ of
$T_n(A)$ exists as an absolutely continuous measure on $\Gamma_0$. An explicit
expression for the measure $\mu_0$ was obtained by Hirschman \cite{Hirschman}.
An alternative expression for $\mu_0$ is given by \eqref{def:measuresk} below
with $k=0$, cf.\ \cite{DK}. Further results about $\mu_0$ in the scalar case
$r=1$ can be found in \cite{BG,BS2,DK,Hirschman,Ullman}.

For the case of banded block Toeplitz matrices, $r>1$, Widom \cite{Widom1}
showed that the above results remain essentially valid, provided that the
following hypotheses H2 and H3 hold true. The hypothesis H1 is stated for
further reference.

\begin{itemize}
\item[H1.] The set $\Lam$ in \eqref{def:laml} is empty.
\item[H2.] The set $\Gamma_0$ in \eqref{def:curves0} is a subset of $\cee$ of $2$-dimensional Lebesgue measure zero.
\item[H3.] The set $G_0$ in
\eqref{def:tilGamma0} below has finite cardinality.
\end{itemize}

In the hypothesis H3 we define the set \begin{equation}\label{def:tilGamma0}
G_0 := \{\lam\in\cee\setminus\Gamma_0\mid C_{0}(\lam)=0\},\end{equation}  with
\begin{equation}\label{def:C0} C_{0}(\lam) := \det\left( \frac{1}{2\pi
i}\int_{\sigma_0} z^{\mu-\nu} (A(z)-\lam I_r)^{-1}\frac{dz}{z}
\right)_{\mu,\nu=1,\ldots,\alpha},\qquad\textrm{for}\
\lam\in\cee\setminus\Gamma_0,
\end{equation}
with $\alpha$ in \eqref{def:A}, and where $\sigma_0$ is a counterclockwise
oriented closed Jordan curve enclosing $z=0$ and the points $z_j(\lam)$,
$j=1,\ldots,q$, but no other roots of $f(z,\lam)=0$. In \eqref{def:C0} the
determinant is taken of a matrix of size $r\alpha$ by $r\alpha$ and the
integral is defined entry-wise. For background, generalizations and alternative
representations for the function $C_{0}(\lam)$ we refer to
\cite{BS2,BS,Widom1,Widom2} ($C_0$ corresponds to the function $E[\varphi]$ in
\cite{Widom1,Widom2}),
see also Prop.~\ref{proposition:Widomformula} below.

Widom shows that under the above hypotheses H2 and H3, one has that
\begin{equation}\label{ss:block} \liminf_{n\to\infty} \sp T_n(A) = \limsup_{n\to\infty} \sp
T_n(A) = \Gamma_0\cup G_0.
\end{equation}
It can be shown that $G_0$ is empty in the scalar case $r=1$, and then
\eqref{ss:block} reduces to \eqref{ss:scalar}. Under H2 and H3, Widom also
observes that Hirschman's expression \cite{Hirschman} for the limiting
eigenvalue distribution $\mu_0$ remains valid.

If hypothesis H1 fails then the limiting eigenvalue distribution of $T_n(A)$
may contain point masses. This is implicit in Widom \cite{Widom1} and will be
described in detail in this paper. On the other hand, if H2 fails then Widom's
results are not true in the stated form. Usually they remain valid in a
modified form however, see Sections~\ref{subsection:Gammak2D} and
\ref{subsection:exampledegenerate} below.

The failure of hypothesis H3 is more serious, and it may cause the results to
break down (see e.g.\ Section~\ref{subsection:exampledegenerateII}). Therefore
it is important to provide sufficient conditions guaranteeing that H3 holds
true. We will provide two such conditions; in both cases H2 will hold true as
well.

\begin{proposition}\label{prop:twoconditions}(Sufficient conditions for H2 and
H3).
\begin{itemize}
\item[(a)] Suppose that the set $\cee\setminus\Gamma_0$ is connected and moreover
$\Gamma_0$ does not have any interior points. Then H2 and H3 hold true.
\item[(b)] Suppose that $A(z)$ is the symbol of a lower Hessenberg matrix, in the sense
that in the entry-wise expansion $T_n(A) = (t_{i,j})_{i,j=0}^{rn-1}$ we have
$t_{i,j}=0$ whenever $j>i+1$, i.e., all the entries above the first scalar
superdiagonal of $T_n(A)$ vanish. Then H2 (or more generally H2$k$ below) and
H3 hold true.
\end{itemize}
\end{proposition}
Proposition~\ref{prop:twoconditions} will be proved in
Section~\ref{subsection:proof:proptwo}. Incidentally, the assumption
\eqref{assumption:pq} implies that all the entries on the first scalar
superdiagonal of the Hessenberg matrix $T_n(A)$ in Part~(b) are
non-zero.\smallskip

Finally we discuss the results of Duits-Kuijlaars \cite{DK}. These authors
noticed that in addition to the set $\Gamma_0$ in \eqref{def:curves0}, an
important role is played by the sets
\begin{equation}\label{def:Gammak} \Gamma_k := \{\lam\in\cee\mid |z_{q+k}(\lam)| = |z_{q+k+1}(\lam)|
\},
\end{equation}
for $k=-q+1,\ldots,p-1$. In the scalar case $r=1$, each set $\Gamma_k$ is a
curve consisting of finitely many analytic arcs. We equip every analytic arc of
$\Gamma_k$ with an orientation and we define the $+$-side (or $-$-side) as the
side on the left (or right) of the arc when traversing it according to its
orientation.

Duits and Kuijlaars then define the measure
\begin{equation} \label{def:measuresk} d\mu_k(\lam) = \frac{1}{2\pi i}\sum_{j=1}^{q+k}\left(
\frac{z_{j+}'(\lam)}{z_{j+}(\lam)}-\frac{z_{j-}'(\lam)}{z_{j-}(\lam)}
\right)d\lam
\end{equation}
on the curve $\Gamma_k$. Here $d\lam$ denotes the complex line element on each
analytic arc of $\Gamma_k$, according to the chosen orientation of $\Gamma_k$.
In addition, $z_{j+}(\lam)$ and $z_{j_-}(\lam)$ denote the boundary values of
$z_j(\lambda)$ from the $+$-side and $-$-side of $\Gamma_k$, respectively.
These boundary values exist for all but finitely many points. The definition
\eqref{def:measuresk} is independent of the choice of the orientation of
$\Gamma_k$.

In the scalar case $r=1$, it is shown in \cite{DK} that the measures $\mu_k$
are the minimizers of a certain (vector) equilibrium problem from potential
theory. Moreover, $\mu_k$ is the weak limit for $n\to\infty$ of the normalized
counting measures of the \emph{$k$th generalized eigenvalues} of the Toeplitz
matrix $T_n(A)$. The usual eigenvalues correspond to $k=0$.

In this paper we wish to extend these results to the block case $r\geq 2$.
Instead of hypothesis H2 we are then led to the following generalization H2$k$:
\begin{flushleft}
H2$k$. Each set $\Gamma_k$ in \eqref{def:Gammak}, $k=-q+1,\ldots,p-1$, is a
subset of $\cee$ of 2-dimensional Lebesgue measure zero.
\end{flushleft}
To the algebraic curve $f(z,\lam)=0$ we will associate an equilibrium problem,
even when hypotheses H1 and/or H2$k$ fail. Then the equilibrium problem may
contain point sources (if H1 fails) or the definition of $\Gamma_k$ in
\eqref{def:Gammak} needs to be modified (if H2$k$ fails).

The measure $\mu_0$ will be one of the measures involved in the equilibrium
problem. This measure will be the absolutely continuous part of the limiting
eigenvalue distribution of $T_n(A)$, provided that hypothesis H3, or a suitable
analogue thereof if H2 fails, holds true. In particular this will be the case
for the two situations in Prop.~\ref{prop:twoconditions}. There is also an
interpretation of the measure $\mu_k$, $k\neq 0$, as the absolutely continuous
part of the limiting distribution of the $k$th generalized eigenvalues of
$T_n(A)$, in the spirit of \cite{DD,DK}; this will be briefly discussed in
Section~\ref{subsection:geneig}.

In the next section, we associate an equilibrium problem to an \emph{arbitrary}
algebraic curve $f(z,\lam)=0$ as in \eqref{f:expansion}, which is not
necessarily defined from a block Toeplitz matrix. In
Section~\ref{section:equiltoep} we apply this to banded block Toeplitz matrices
$T_n(A)$. In Section~\ref{section:casestudy:scalarband} we specialize our
results to the case where $T_n(A)$ has a scalar banded structure.
Section~\ref{section:proofs} contains the proofs of our main results.
Section~\ref{section:examples} illustrates our results for some examples.
Finally, Section~\ref{section:conclusion} contains some concluding remarks.

\section{Equilibrium problem associated to an arbitrary algebraic curve}
\label{section:equil:arbitrarycurve}

\subsection{Definitions} \label{subsection:mk}

In this section we show how an equilibrium problem can be associated to an
arbitrary algebraic curve. We consider an algebraic curve which is written in
the form
\begin{equation}\label{f:expansion:arb} f(z,\lam) =
\sum_{k=-q}^{p} f_k(\lam) z^k = 0,
\end{equation}
where $f_k(\lam)$, $k=-q,\ldots,p$ are polynomials, and where $p,q\geq 1$ are
such that the outermost polynomials $f_{-q}(\lam)$ and $f_p(\lam)$ are not
identically zero. Note that the numbers $p$ and $q$ in \eqref{f:expansion:arb}
do not have an absolute meaning; indeed by multiplying $f$ with $z^j$,
$j\in\zet$, the indices $p$ and $q$ are shifted to $p+j$ and $q+j$
respectively. The reason why we write \eqref{f:expansion:arb} in its present
form is because of the applications to banded block Toeplitz matrices.


Denote
\begin{equation}\label{def:degreepol}
r:=\max_{k\in\{-q,\ldots,p\}} \deg f_k.
\end{equation}
This definition of $r$ is compatible with the one used before, by virtue of
\eqref{f:expansionlargelambda}.


Define the roots $z_j=z_j(\lam)$, $j=1,\ldots,p+q$ as in
\eqref{ordering:roots}, and define the sets $\Gamma_k$, $k=-q+1,\ldots,p-1$ as
in \eqref{def:Gammak}. The structure of the set $\Gamma_k$ is given by the
following result.

\begin{lemma}\label{lemma:structureGammak} (Structure of $\Gamma_k$).
Let $k\in\{-q+1,\ldots,p-1\}$. Then any point $\lam_0\in\cee$ has an open
neighborhood $U\subset\cee$ whose intersection with $\Gamma_k$ is either empty,
the singleton $\{\lam_0\}$, the entire neighborhood $U$, or a finite union of
analytic arcs moving from $\lam_0$ to the boundary $\partial U$ of the
neighborhood $U$, with the arcs intersecting only at the point $\lam_0$. A
similar statement holds true for $\lam_0=\infty$ provided that we consider
$\Gamma_k$ on the Riemann sphere $\overline{\cee}:=\cee\cup\{\infty\}$. The
isolated points of $\Gamma_k$ all belong to $\Lam$ in \eqref{def:laml}.
\end{lemma}

Lemma~\ref{lemma:structureGammak} was observed for $k=0$ by Widom \cite[Page
312]{Widom1}, based on the similar result for the scalar case by Schmidt and
Spitzer \cite{SchmidtSpitzer}. The proof for $k\neq 0$ is exactly the~same. See
also Prop.~\ref{prop:componentsGammak} below for further information on
$\Gamma_k$.


In addition to the set $\Gamma_k$ we also introduce
\begin{eqnarray}\label{def:Gammak:wtil} \wtil\Gamma_k &:=&
\Gamma_k\setminus\{\textrm{isolated points of $\Gamma_k$}\} \\
\nonumber &=& \cls\{\lam\in\cee\setminus\Lam\mid |z_{q+k}(\lam)| =
|z_{q+k+1}(\lam)| \},
\end{eqnarray}
for $k=-q+1,\ldots,p-1$, where $\cls$ denotes the closure of a subset of
$\cee$.

Our next goal is to provide an expression for the total mass of the measure
$\mu_k$ in \eqref{def:measuresk}. To this end we need some auxiliary
definitions. The next definition is a variant of the so-called \emph{Newton
polygon}, see e.g.\ \cite{Goss}.

\begin{definition}\label{def:mk} (The numbers $m_k$).
We denote by $k\mapsto m_k$ the smallest concave function on $\{-q,\ldots,p\}$
for which $m_k\geq \deg f_k$ for all $k$. Formally,
$$ m_k = \max_{i\leq k\leq j} \left(\frac{j-k}{j-i}\deg f_i+\frac{k-i}{j-i}\deg f_j\right),
$$
where the maximum is taken over all integers $i,j$ with $-q\leq i\leq k$ and
$k\leq j\leq p$ and with the equalities $i=k$ and $j=k$ not holding
simultaneously.
\end{definition}

A graphical interpretation of Definition~\ref{def:mk} is as follows: consider
the grid points $(k,m_k)\in\zet^2$, $k=-q,\ldots,p$, and draw a line segment
between $(k,m_k)$ and $(k+1,m_{k+1})$, for $k=-q,\ldots,p-1$. This then results
in a curve which lies above the grid points $(k,\deg f_k)\in\zet^2$, and which
is the \lq lowest\rq\ concave, piecewise linear curve with this property.

Let us illustrate Definition~\ref{def:mk} for two examples.

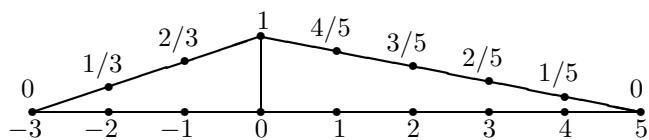
\begin{figure}[htb]
\vspace{-54mm}
\begin{center}
   \setlength{\unitlength}{1truemm}
   \begin{picture}(100,70)(-5,2)

       \put(0,0){\thicklines\circle*{1}}
       \put(-1.5,2){0}
       \put(10,0){\thicklines\circle*{1}}
       \put(10,3.333){\thicklines\circle*{1}}
       \put(6.5,5.333){1/3}
       \put(20,0){\thicklines\circle*{1}}
       \put(20,6.667){\thicklines\circle*{1}}
       \put(16.5,8.667){2/3}
       \put(30,0){\thicklines\circle*{1}}
       \put(30,10){\thicklines\circle*{1}}
       \put(29.5,11){1}
       \put(40,0){\thicklines\circle*{1}}
       \put(40,8){\thicklines\circle*{1}}
       \put(36.5,10.0){4/5}
       \put(50,0){\thicklines\circle*{1}}
       \put(50,6.0){\thicklines\circle*{1}}
       \put(46.5,8.0){3/5}
       \put(60,0){\thicklines\circle*{1}}
       \put(60,4){\thicklines\circle*{1}}
       \put(56.5,6){2/5}
       \put(70,0){\thicklines\circle*{1}}
       \put(70,2){\thicklines\circle*{1}}
       \put(66.5,4){1/5}
       \put(80,0){\thicklines\circle*{1}}
       \put(78.5,2){0}

       \put(0,0){\line(1,0){80}}
       \put(0,0){\thicklines\line(3,1){30}}
       \put(30,10){\thicklines\line(5,-1){50}}
       \put(30,0){\line(0,1){10}}

       \put(-3.3,-3.3){$-3$}
       \put(6.7, -3.3){$-2$}
       \put(16.7,-3.3){$-1$}
       \put(29.2,-3.3){$0$}
       \put(39.5,-3.3){$1$}
       \put(49.2,-3.3){$2$}
       \put(59.2,-3.3){$3$}
       \put(69.2,-3.3){$4$}
       \put(79.2,-3.3){$5$}

  \end{picture}
   \vspace{3mm}
   \caption{Illustration for Example~\ref{example:mk1}.}
   \label{fig:newton0}
\end{center}
\end{figure}

\begin{example}\label{example:mk1} For the situation in Duits-Kuijlaars \cite{DK}
we have $\deg f_k = 1$ if $k=0$ and $\deg f_k = 0$ if $k\neq 0$. Then we easily
find that
\begin{equation}\label{mk:Duits} (m_k)_{k=-q}^p = \left(0,\frac 1q,\frac
2q,\ldots,1,\ldots,\frac 2p,\frac 1p,0\right).
\end{equation}
Let us illustrate this if $q=3$ and $p=5$. In that case $(\deg
f_k)_{k=-3}^{5}=(0,0,0,1,0,0,0,0,0)$, and Figure~\ref{fig:newton0} shows how to
construct the concave, piecewise linear curve lying above the grid points
$(k,\deg f_k)$. From the figure we can then read off that $(m_k)_{k=-3}^{5} =
(0,\frac 13,\frac 23,1,\frac 45,\frac 35,\frac 25,\frac 15,0)$. Finally, we
observe that the number $m_k$ in \eqref{mk:Duits}, $k=-q+1,\ldots,p-1$, is
precisely the total mass of the measure $\mu_k$ in~\cite{DK}.
\end{example}

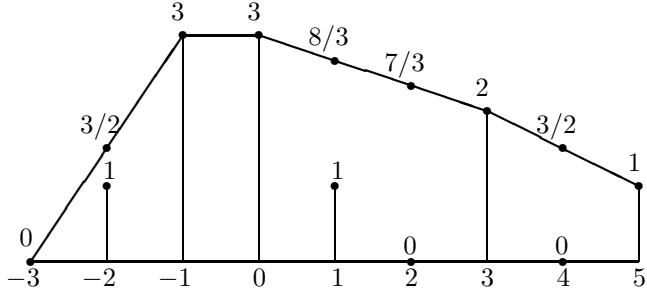
\begin{figure}[htb]
\vspace{-34mm}
\begin{center}
   \setlength{\unitlength}{1truemm}
   \begin{picture}(100,70)(-5,2)

       \put(0,0){\thicklines\circle*{1}}
       \put(-1.5,2){0}
       \put(10,10){\thicklines\circle*{1}}
       \put(9.5,11){1}
       \put(10,15){\thicklines\circle*{1}}
       \put(6.5,17){3/2}
       \put(20,30){\thicklines\circle*{1}}
       \put(18.5,32){3}
       \put(30,30){\thicklines\circle*{1}}
       \put(28.5,32){3}
       \put(40,10){\thicklines\circle*{1}}
       \put(39.5,11){1}
       \put(40,26.667){\thicklines\circle*{1}}
       \put(36.5,28.667){8/3}
       \put(50,0){\thicklines\circle*{1}}
       \put(49.0,1){0}
       \put(50,23.333){\thicklines\circle*{1}}
       \put(46.5,25.333){7/3}
       \put(60,20){\thicklines\circle*{1}}
       \put(58.5,22){2}
       \put(70,0){\thicklines\circle*{1}}
       \put(70,15){\thicklines\circle*{1}}
       \put(66.5,17){3/2}
       \put(69.0,1){0}
       \put(80,10){\thicklines\circle*{1}}
       \put(78.5,12){1}

       \put(0,0){\line(1,0){80}}
       \put(20,30){\thicklines\line(1,0){10}}
       \put(0,0){\thicklines\line(2,3){20}}
       \put(30,30){\thicklines\line(3,-1){30}}
       \put(60,20){\thicklines\line(2,-1){20}}
       \put(10,0){\line(0,1){10}}
       \put(20,0){\line(0,1){30}}
       \put(30,0){\line(0,1){30}}
       \put(40,0){\line(0,1){10}}
       \put(60,0){\line(0,1){20}}
       \put(80,0){\line(0,1){10}}

       \put(-3.3,-3.3){$-3$}
       \put(6.7, -3.3){$-2$}
       \put(16.7,-3.3){$-1$}
       \put(29.2,-3.3){$0$}
       \put(39.5,-3.3){$1$}
       \put(49.2,-3.3){$2$}
       \put(59.2,-3.3){$3$}
       \put(69.2,-3.3){$4$}
       \put(79.2,-3.3){$5$}

  \end{picture}
   \vspace{3mm}
   \caption{Illustration for Example~\ref{example:mk2}.}
   \label{fig:newton1}
\end{center}
\end{figure}

\begin{example}\label{example:mk2} Assume that $q=3$, $p=5$ and
$(\deg f_k)_{k=-3}^{5}=(0,1,3,3,1,0,2,0,1)$. Proceeding in a similar way as
before, we find that
$$(m_k)_{k=-3}^{5} = \left(0,\frac 32,3,3,\frac 83,\frac 73,2,\frac
32,1\right),$$ as illustrated in Figure~\ref{fig:newton1}.
\end{example}

Recall the definition of the set $\Lam$ in \eqref{def:laml}, and choose an
arbitrary but fixed labeling of the elements of this set, i.e.,
$\Lambda=:\{\lam_l\}_{l=1}^L$. Note that under hypothesis H1 we have that
$L:=|\Lam|=0$, so in that case we can ignore all the arguments involving the
numbers $(\lam_l)_{l=1}^L$ in what follows.

We need the following analogue of Definition~\ref{def:mk}.

\begin{definition}\label{def:mkl} (The numbers $m_k^{(l)}$).
Fix $l\in\{1,\ldots,L\}$. We denote by $k\mapsto m_k^{(l)}$ the largest convex
function on $\{-q,\ldots,p\}$ for which $m_k^{(l)}\leq \mult_{\lam-\lam_l} f_k$
for all $k$, where $\mult_{\lam-\lam_{l}}f_k$ denotes the multiplicity of
$\lam-\lam_{l}$ as a factor of $f_k(\lam)$. Formally,
$$ m_k^{(l)} = \min_{i\leq k\leq j} \left(\frac{j-k}{j-i}\mult_{\lam-\lam_{l}}f_i+\frac{k-i}{j-i}\mult_{\lam-\lam_{l}}f_j\right),
$$
where the minimum is taken over all integers $i,j$ with $-q\leq i\leq k$ and
$k\leq j\leq p$ and with the equalities $i=k$ and $j=k$ not holding
simultaneously.
\end{definition}

A graphical interpretation of Definition~\ref{def:mkl} is as follows: consider
the grid points $(k,m_k^{(l)})~\in~\zet^2$, $k=-q,\ldots,p$, and draw a line
segment between $(k,m_k^{(l)})$ and $(k+1,m_{k+1}^{(l)})$, for
$k=-q,\ldots,p-1$. This then results in a curve which lies below the grid
points $(k,\mult_{\lam-\lam_l} f_k)\in\zet^2$, and which is the \lq highest\rq\
convex, piecewise linear curve with this property.

\begin{figure}[htb]
\vspace{-44mm}
\begin{center}
   \setlength{\unitlength}{1truemm}
   \begin{picture}(100,70)(-5,2)

       \put(0,10){\thicklines\circle*{1}}
       \put(-2.0,12.5){1}
       \put(0,10){\line(0,1){20}}
       \put(10,10){\thicklines\circle*{1}}
       \put(10,10){\line(0,1){20}}
       \put(10,5){\thicklines\circle*{1}}
       \put(8,7){1/2}
       \put(8.0,12.5){1}
       \put(20,0){\thicklines\circle*{1}}
       \put(20,0){\line(0,1){30}}
       \put(18.0,2){0}
       \put(30,0){\thicklines\circle*{1}}
       \put(30,20){\thicklines\circle*{1}}
       \put(28.0,22.5){2}
       \put(28.0,2){0}
       \put(30,20){\line(0,1){10}}
       \put(40,0){\thicklines\circle*{1}}
       \put(40,10){\thicklines\circle*{1}}
       \put(38.0,12.5){1}
       \put(38.0,2){0}
       \put(40,10){\line(0,1){20}}
       \put(50,0){\thicklines\circle*{1}}
       \put(50,0){\line(0,1){30}}
       \put(48.0,2){0}
       \put(60,5){\thicklines\circle*{1}}
       \put(58,7){1/2}
       \put(60,10){\thicklines\circle*{1}}
       \put(58.0,12.5){1}
       \put(60,10){\line(0,1){20}}
       \put(70,10){\thicklines\circle*{1}}
       \put(70,10){\line(0,1){20}}
       \put(68.0,12.5){1}
       \put(80,20){\thicklines\circle*{1}}
       \put(80,20){\line(0,1){10}}
       \put(78.0,22.5){2}

       \put(0,10){\thicklines\line(2,-1){20}}
       \put(20,0){\thicklines\line(1,0){30}}
       \put(50,0){\thicklines\line(2,1){20}}
       \put(70,10){\thicklines\line(1,1){10}}

       \put(0,0){\line(1,0){80}}
       \put(-3.3,-3.3){$-3$}
       \put(6.7, -3.3){$-2$}
       \put(16.7,-3.3){$-1$}
       \put(29.2,-3.3){$0$}
       \put(39.5,-3.3){$1$}
       \put(49.2,-3.3){$2$}
       \put(59.2,-3.3){$3$}
       \put(69.2,-3.3){$4$}
       \put(79.2,-3.3){$5$}

  \end{picture}
   \vspace{7mm}
   \caption{Illustration for Example~\ref{example:mk3}.}
   \label{fig:newton2}
\end{center}
\end{figure}
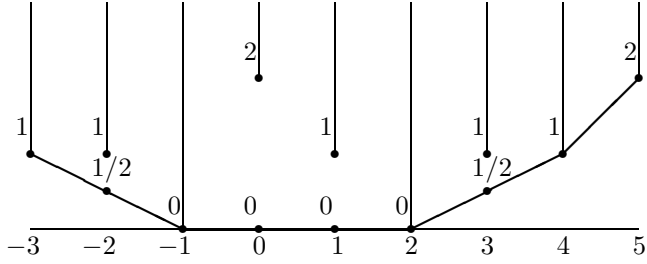

\begin{example}\label{example:mk3}
Assume that $q=3$, $p=5$ and suppose that $\lam_l\in\Lam$ is such that
$(\mult_{\lam-\lam_{l}} f_k)_{k=-3}^{5}=(1,1,0,2,1,0,1,1,2)$.
Figure~\ref{fig:newton2} then shows how to construct the convex, piecewise
linear curve lying below the grid points $(k,\mult_{\lam-\lam_{l}} f_k)$. We
can then read off that $$(m_k^{(l)})_{k=-3}^{5} = \left(1,\frac
12,0,0,0,0,\frac 12,1,2\right).$$
\end{example}

\begin{definition}\label{def:tilmk} (The numbers $\til m_k$).
We define
\begin{equation}
\label{totalsumofmasses} \til m_k:= m_k - \left(m_{k}^{(1)}+\ldots+
m_{k}^{(L)}\right),
\end{equation}
for $k=-q+1,\ldots,p-1$, where $m_k$ and $m_k^{(l)}$ are as in
Definitions~\ref{def:mk} and \ref{def:mkl}.
\end{definition}

The number $\til m_k$ in \eqref{totalsumofmasses} will be the total mass of the
measure $\mu_k$ in \eqref{def:measuresk}. Note that we defined $\til m_k$ for
$k=-q+1,\ldots,p-1$. We could also define $\til m_k$ for $k=-q$ or $k=p$, by
using the same definition \eqref{totalsumofmasses}. But in that case it is easy
to see that $\til m_{-q}=\til m_p = 0$.

\begin{lemma}\label{lemma:mkpos}
The numbers $\til m_k$ in \eqref{totalsumofmasses} satisfy
\begin{equation}
\label{totalsumofmasses:pos} \til m_k\geq 0,
\end{equation}
for any $k=-q+1,\ldots,p-1$. These inequalities are strict if
\eqref{f:expansionlargelambda} holds.
\end{lemma}

\begin{lemma}\label{lemma:equality} (Criteria for $\til m_k=0$).
In Lemma~\ref{lemma:mkpos}, the following statements give equivalent conditions
for having equality in \eqref{totalsumofmasses:pos}:
\begin{enumerate} \item[(i)] $\til m_k=0$
for some $k\in\{-q+1,\ldots,p-1\}$, \item[(ii)] $\til m_k=0$ for all
$k=-q,\ldots,p$, \item[(iii)] $f(z,\lam) = g(\lam)\til f(h(\lam)z),$ where $g$
is a polynomial and $h$ a rational function of $\lam$, and where $\til
f(z)=c_{-q}z^{-q}+\cdots+c_pz^p$ is a Laurent polynomial with coefficients
$c_k$ \emph{not} depending on~$\lam$.
\end{enumerate}
\end{lemma}

Lemmas~\ref{lemma:mkpos} and ~\ref{lemma:equality} are proved in
Section~\ref{subsection:proof:lemmaremark}. The nonnegativity of $\til m_k$ can
also be deduced from the fact that it is the total mass of the positive measure
$\mu_k$. To avoid trivial statements, we will often tacitly assume that none of
the equivalent conditions in Lemma~\ref{lemma:equality} is satisfied.

Here is an addendum to Lemma~\ref{lemma:structureGammak}:

\begin{proposition}\label{prop:componentsGammak} (Connected components of
$\Gamma_k$). Suppose that hypothesis H2$k$ holds true. Then the number of
compact, connected components of $\Gamma_k$ is $\leq m_k$ (and hence $\leq r$).
Moreover, for each compact, connected component $C$ of $\Gamma_k$, denote by
$\mu_k(C)$ the total mass of the restriction of the measure $\mu_k$ in
\eqref{def:measuresk} to $C$, with $\mu_k(C)=0$ if $C$ is an isolated point of
$\Gamma_k$. Then we have that
\begin{equation}\label{componentsGammak:H1} \mu_k(C)\in\enn,\qquad\textrm{if hypothesis H1 holds,}
\end{equation}
and in general,
\begin{equation}\label{componentsGammak}
\mu_k(C)+\sum_{\lam_l\in \Lam\cap C} m_k^{(l)}\in\enn,\end{equation} where the
sum runs over all $l=1,\ldots,L$ with $\lam_l\in\Lam\cap C$.
\end{proposition}

Prop.~\ref{prop:componentsGammak} will be proved in
Section~\ref{subsection:proof:theorem:main1}. The proposition was observed
before by Widom \cite[Page 315]{Widom1} in the special case where $k=0$ and
$m_0^{(l)}=0$ for all $l$; note that \eqref{componentsGammak} then reduces to
\eqref{componentsGammak:H1}. In the scalar case $r=1$ it generalizes Ullman's
result~\cite{Ullman} that $\Gamma_0$ is connected.

\subsection{The equilibrium problem}
\label{subsection:equil1}

Now we associate an equilibrium problem to \eqref{f:expansion:arb}. First we do
this under the hypothesis H2$k$. We will closely follow \cite{DD,DK}. For any
measure $\mu$ on $\cee$ define its logarithmic energy as
\begin{equation*}\label{deflogarithmicenergybis}
I(\mu) = \int_{}\int_{}\ \log\frac{1}{|x-y|} \ d\mu(x)\ d\mu(y).
\end{equation*}
Similarly, for any measures $\mu,\nu$ on $\cee$ define their mutual energy as
\begin{equation*}\label{defmutualenergybis}
I(\mu,\nu) = \int_{}\int_{}\ \log\frac{1}{|x-y|} \ d\mu(x)\ d\nu(y).
\end{equation*}

\begin{definition}\label{def:admissible}
We call a vector of positive measures $\vec{\nu} =
(\nu_{-q+1},\ldots,\nu_{p-1})$ \emph{admissible} if $\nu_k$ has finite
logarithmic energy, $\nu_k$ is supported on $\wtil\Gamma_k$ in
\eqref{def:Gammak:wtil}, and it has total mass $\nu_k(\wtil\Gamma_k) = \til
m_k$ for every $k=-q+1,\ldots,p-1$, recall
\eqref{totalsumofmasses}--\eqref{totalsumofmasses:pos}.
\end{definition}

\begin{definition}\label{def:equil}
The \emph{energy functional} $J$ is defined by
\begin{multline}\label{energyfunctional} J(\vec{\nu}) = \sum_{k=-q+1}^{p-1}
I(\nu_k) - \sum_{k=-q+1}^{p-2}
I(\nu_k,\nu_{k+1})\\+\sum_{l=1}^{L}\sum_{k=-q+1}^{p-1}
(-m_{k-1}^{(l)}+2m_{k}^{(l)}-m_{k+1}^{(l)})\int\log\frac{1}{|\lam-\lam_{l}|}\
d\nu_k(\lam).
\end{multline}
The \emph{(vector) equilibrium problem} is to minimize the energy functional
\eqref{energyfunctional} over all admissible vectors of positive measures
$\vec{\nu}$.
\end{definition}

Note that the numbers $-m_{k-1}^{(l)}+2m_{k}^{(l)}-m_{k+1}^{(l)}$ in
\eqref{energyfunctional} are all nonpositive because of the convexity of
$k\mapsto m_k^{(l)}$.

The equilibrium problem can be understood intuitively as follows, compare with
\cite{DD,DK}. On each of the curves $\wtil\Gamma_k$ (recall the assumption
H2$k$) we put charged particles with total charge $\til m_k$. The particles on
each curve repel each other. The particles on two consecutive curves attract
each other, with a strength that is half as strong as the repulsion on each
individual curve. Particles on different curves that are non-consecutive do not
interact directly. Moreover, if $2m_k^{(l)}\neq m_{k-1}^{(l)}+m_{k+1}^{(l)}$
then we have an attracting external field acting on the particles on the curve
$\wtil\Gamma_{k}$. The external field is induced by an attracting point charge
(also called \emph{sink}) at $\lam=\lam_l$. We refer to~\cite{Rans,SaffTotik}
for
background on equilibrium problems with external fields, and 
to \cite{NS} for vector equilibrium problems.

Note that if hypothesis H1 holds true then \eqref{energyfunctional} reduces to
\begin{equation}\label{energyfunctional:H1} J(\vec{\nu}) = \sum_{k=-q+1}^{p-1}
I(\nu_k) - \sum_{k=-q+1}^{p-2} I(\nu_k,\nu_{k+1}).
\end{equation}
This is the energy functional in \cite{DK}; it also appears in the theory of
Nikishin systems \cite{NS}.

The following theorem generalizes a result in \cite{DD} and \cite{DK}.

\begin{theorem}\label{theorem:main1} (Equilibrium problem associated to an algebraic curve). Consider an
algebraic curve as in \eqref{f:expansion:arb} and define the sets $\Gamma_k$,
$\wtil\Gamma_k$, $k=-q+1,\ldots,p-1$ as in \eqref{def:Gammak} and
\eqref{def:Gammak:wtil}. Assume that hypothesis H2$k$ holds true. Then
\begin{itemize}
\item[(a)] The vector of measures $\vec{\mu} = (\mu_{k})_{k=-q+1}^{p-1}$ defined in \eqref{def:measuresk} is admissible.
\item[(b)] There
exist constants $l_k\in\mathbb{R}$ such that \begin{multline}
2\int\log\frac{1}{|\lam-x|}\ d\mu_k(x) - \int\log\frac{1}{|\lam-x|}\
d\mu_{k+1}(x) - \int\log\frac{1}{|\lam-x|}\ d\mu_{k-1}(x)\\
+\sum_{l=1}^L(-m_{k-1}^{(l)}+2m_{k}^{(l)}-m_{k+1}^{(l)})\log\frac{1}{|\lam-\lam_l|}
 = l_k,
\end{multline}
for $\lam\in\wtil\Gamma_k$, $k\in\{-q+1,\ldots,p-1\}$. Here we let $\mu_{-q}$
and $\mu_{p}$ be the zero measures.
\item[(c)] $\vec{\mu} =(\mu_{k})_{k=-q+1}^{p-1}$ is the unique solution to the equilibrium
problem in Def.~\ref{def:equil}.
\end{itemize}
\end{theorem}

Theorem~\ref{theorem:main1} will be proved in
Section~\ref{subsection:proof:theorem:main1}. Note that the equalities in
Part~(b) are nothing but the Euler-Lagrange variational conditions of the
equilibrium problem, see also \cite{DD,DK}.

\subsection{Roots with identically equal modulus}
\label{subsection:Gammak2D}

Now we extend Theorem~\ref{theorem:main1} to the case where hypothesis~H2$k$
fails, i.e., the case where one or more sets $\Gamma_k$ have non-zero
2-dimensional Lebesgue measure in $\cee$. By Lemma~\ref{lemma:structureGammak}
this implies that $\Gamma_k$ contains an open disk $U$. Inside this disk, two
or more roots $z(\lam)$ of $f(z,\lam)$ have identically equal modulus as
functions of $\lam$. If the disk $U$ is disjoint from $\Lam$, then we can label
these roots so that they are analytic functions in $U$. The maximum modulus
principle then implies that they are identically equal as functions of $\lam$,
up to a constant factor of modulus~$1$.

In this case we adapt the definition of the sets $\Gamma_k$ as follows:
\begin{multline}\label{def:Gammak:mod} \Gamma_k = \{\lam\in\cee\setminus\Lam\mid |z_{q+k}(\lam)| = |z_{q+k+1}(\lam)|
\textrm{ and, possibly after relabeling the roots, the} \\
\textrm{function $z_{q+k}/z_{q+k+1}$ takes infinitely many values on each open
neighborhood $U$ of $\lam$}\}\\ \cup\ \{\lam_l\in\Lam\mid m_k^{(l)}>0\}.
\end{multline}
This new definition guarantees that $\Gamma_k$ is a curve:

\begin{lemma}\label{lemma:strucwk} Fix $k\in\{-q+1,\ldots,p-1\}$.
\begin{itemize}\item[(a)] The set $\Gamma_k$ in
\eqref{def:Gammak:mod} is a finite union of analytic arcs and points, with all
of its isolated points belonging to $\Lam$, and it satisfies
Lemma~\ref{lemma:structureGammak}.
\item[(b)] For any simply connected domain
$U\subset\cee\setminus(\Gamma_k\cup\Lam)$, we can choose an ordering of the
roots $z_j(\lam)$ as in \eqref{ordering:roots} such that $\prod_{j=1}^{q+k}
z_j(\lam)$ is analytic for $\lam\in U$. Moreover, we can uniquely define the
logarithmic derivative $\left(\prod_{j=1}^{q+k}
z_j(\lam)\right)'/\prod_{j=1}^{q+k} z_j(\lam)$ as a meromorphic function in
$\cee\setminus\Gamma_k$ with poles at the points in $\Lam$.

\end{itemize}
\end{lemma}

Lemma~\ref{lemma:strucwk} is proved in
Section~\ref{subsection:proof:lemmamerom}.

Due to Lemma~\ref{lemma:strucwk}, we can uniquely define the measure $\mu_k$ on
$\Gamma_k$ (more precisely on $\wtil\Gamma_k$) by means of
\eqref{def:measuresk}. We have the following generalization of Theorem
\ref{theorem:main1}.

\begin{theorem}\label{theorem:main1:mod} (Equilibrium problem with roots of identically equal
modulus). Consider the setting of Theorem~\ref{theorem:main1} but assume that
hypothesis~H2$k$ fails. Define the curves $\Gamma_k$, $\wtil\Gamma_k$ as in
\eqref{def:Gammak:mod} and \eqref{def:Gammak:wtil} and the measures $\mu_k$ as
in \eqref{def:measuresk}, taking into account Lemma~\ref{lemma:strucwk}. Then
Theorem \ref{theorem:main1} remains valid.
\end{theorem}

This theorem is proved in Section~\ref{subsection:proof:theorem:main1}.

\section{The measure $\mu_0$ as the limiting eigenvalue distribution of
the banded block Toeplitz matrix $T_n(A)$} \label{section:equiltoep}

Using the results of the previous section, we can associate a vector
equilibrium problem to the algebraic equation $f(z,\lam)=0$ in \eqref{def:f}
that is defined from the banded block Toeplitz matrix $T_n(A)$. We want to show
that the measure $\mu_0$ in the equilibrium problem is the absolutely
continuous part of the limiting distribution of the eigenvalues of $T_n(A)$. As
discussed before, this will require the hypothesis H3 (or a suitable analogue
thereof if H2 fails) to hold true. The next theorem should be compared with
Widom's result~\cite[Theorem~6.1]{Widom1}. We define the normalized eigenvalue
counting measure $\mu_{0,n}$ of $T_n(A)$ as
\begin{align} \label{eq:nun} \mu_{0,n} := \frac{1}{n}\sum_{\lam \in \sp
T_n(A)}\delta_\lambda,
\end{align}
where $\delta_\lam$ is the Dirac measure at $\lambda$ and each eigenvalue is
counted according to its multiplicity.

\begin{theorem}\label{theorem:eigTn} (Limiting eigenvalue distribution of $T_n(A)$). Let $A(z)$ be such that
the assumptions in parts (a) and/or (b) of Prop.~\ref{prop:twoconditions} are
satisfied and define $\Gamma_0$, $G_0$ as in \eqref{def:curves0} and
\eqref{def:tilGamma0}. Then
\begin{equation}\label{main2} \liminf_{n\to\infty} \mathrm{sp}\ T_n(A) =
\limsup_{n\to\infty} \mathrm{sp}\ T_n(A) = \Gamma_0\cup G_0,
\end{equation}
and \begin{equation}\label{weakconvergence} \lim_{n\to\infty} \int \phi(z)\
d\mu_{0,n}(z)= \int \phi(z)\ d\mu_{0}(z)+\sum_{l=1}^L m_{0}^{(l)}\phi(\lam_l)
\end{equation}
for every bounded continuous function $\phi$ on $\cee$.

Moreover, for each $\lam\in G_0$ there is a positive integer $j\in\enn$ (more
precisely, the multiplicity of $\lam$ as a zero of $C_0$) such that for every
sufficiently small open disk $U$ around $\lam$, one has
\begin{equation}\label{weakconvergence:isolated} |U\cap \mathrm{sp}\ T_n(A)| = j,
\end{equation}
for all $n$ sufficiently large, where we take into account eigenvalue
multiplicities.
\end{theorem}

Theorem~\ref{theorem:eigTn} shows that the limiting eigenvalue distribution of
$T_n(A)$ for $n\to\infty$ consists of the absolutely continuous part $\mu_0$
together with a point mass of mass $m_0^{(l)}$ at each $\lam_l\in\Lam$,
$l=1,\ldots,L$. The theorem also shows that $G_0$ attracts \emph{isolated
eigenvalues} in the spectrum of $T_n(A)$. The theorem will be proved in
Section~\ref{subsection:prooftheorem:eigTn}.

It can be checked that $m_0^{(l)}>0$ implies $\lam_l\in\Gamma_0$. The point
$\lam_l$ can then either be an isolated point of $\Gamma_0$ or it can lie on
one or more analytic arcs of $\Gamma_0$.

Incidentally, the occurrence of point masses at the points of
$\{\lam_l\in\Lam\mid m_0^{(l)}>0\}$ can already be seen at the level of the
finite $n$ matrices $T_n(A)$:

\begin{proposition} \label{prop:ptmass:finiten}
Let $A(z)$ in \eqref{def:A} be the symbol of an arbitrary banded block Toeplitz
matrix. Then there exists a constant $c\in\er$ such that
\begin{itemize}\item[(a)] For each $\lam_l\in\Lam$, $l=1,\ldots,L$, we have that
\begin{equation}\label{ptmass:finiten}
\mult_{\lam-\lam_l} \det T_n(A(z)-\lam I_r)\geq m_0^{(l)}n-c,\qquad \textrm{for
all }n\in\enn.
\end{equation}
\item[(b)] We have that
\begin{equation}\label{ptmass:infinity}
\deg_{\lam} \det T_n(A(z)-\lam I_r)\leq m_0n+c,\qquad\textrm{for all }
n\in\enn.
\end{equation}
\end{itemize}
\end{proposition}

Prop.~\ref{prop:ptmass:finiten} is established in
Section~\ref{subsection:prooftheorem:eigTn}.

\subsubsection*{The measure $\mu_k$ and the $k$th generalized eigenvalues of $T_n(A)$: discussion}
\label{subsection:geneig}

Fix $k\in\{-q+1,\ldots,p-1\}$ and define the cyclic shift
matrix \begin{equation}\label{def:Sshiftmx}  S:=\begin{pmatrix} 0 & z \\
I_{r-1} & 0
\end{pmatrix}. \end{equation}
Let $\lam$ be a parameter and consider the \lq shifted\rq\ symbol
\begin{equation}\label{def:A:k} S^{-k}(A(z)-\lam I_r) =:
A_{-\alpha_k}(\lam)z^{-\alpha_k}+\ldots+A_{\beta_k}(\lam) z^{\beta_k},
\end{equation}
for suitable $\alpha_k,\beta_k\in\enn$. We may assume that $\alpha_k,\beta_k$
are such that the coefficients $A_{-\alpha_k}$ and $A_{\beta_k}$ in
\eqref{def:A:k} are not identically zero, although this will not be essential.
Note that for $k=0$ we can take $\alpha_0=\alpha$ and $\beta_0=\beta$ as
in~\eqref{def:A}.

We consider the  \lq shifted\rq\ block Toeplitz matrix $T_n(S^{-k}(A(z)-\lam
I_{r}))$. Note that for $k\geq 0$, this block Toeplitz matrix is obtained from
$T_n(A)-\lam I_{rn}$ by skipping its first $k$ rows and adding $k$ new rows at
the bottom of the matrix, subject to the block Toeplitz structure. A similar
description holds for $k<0$, see also \cite{DD,DK}.

We define the \emph{$k$th generalized spectrum} of $T_n(A)$ as
\begin{equation}\label{geneig:def} \mathrm{sp}_k T_n(A) = \{\lam\in\cee\mid
\det(T_n(S^{-k}(A(z)-\lam I_{r})))=0\}. \end{equation}

\smallskip Inspired by Duits-Kuijlaars \cite{DK}, one may hope to interpret the
measure $\mu_k$, $k\neq 0$, as the absolutely continuous part of the weak limit
of the normalized counting measures of the $k$th generalized eigenvalues of
$T_n(A)$. This limiting distribution should then also have a point mass of mass
$m_k^{(l)}$ at $\lam=\lam_l$, $l=1,\ldots,L$.

It turns out that these ideas can indeed be established, provided that a
suitable analogue H3$k$ of hypothesis H3 holds true. Let us define the
following analogues of the objects $G_0$ and $C_0(\lam)$ in
\eqref{def:tilGamma0}--\eqref{def:C0}:
\begin{equation}\label{def:tilGammak} G_k := \{\lam\in\cee\setminus\Gamma_k\mid
C_{k}(\lam)=0\},\end{equation}  and
\begin{equation}\label{def:Ck} C_{k}(\lam) := \det\left( \frac{1}{2\pi
i}\int_{\sigma_k} z^{\mu-\nu} (A(z)-\lam I_r)^{-1}S^{k}\frac{dz}{z}
\right)_{\mu,\nu=1,\ldots,\alpha_k},\qquad\textrm{for}\
\lam\in\cee\setminus\Gamma_k,
\end{equation}
with $\alpha_k$ in \eqref{def:A:k}, and where $\sigma_k$ is a counterclockwise
oriented closed Jordan curve enclosing $z=0$ and the points $z_j(\lam)$,
$j=1,\ldots,q+k$, but no other roots of $f(z,\lam)=0$. In \eqref{def:Ck} the
determinant is taken of a matrix of size $r\alpha_k$ by $r\alpha_k$ and the
integral is again defined entry-wise.

The hypothesis H3$k$ now reads as follows:
\begin{flushleft}
H3$k$. The set $G_k$ in \eqref{def:tilGammak} has finite cardinality.
\end{flushleft}
Define the normalized counting measure
\begin{align} \label{eq:nun:k} \mu_{k,n} := \frac{1}{n}\sum_{\lam \in  \mathrm{sp}_k
T_n(A)}\delta_\lambda,
\end{align}
where again each root is counted according to its multiplicity.

\begin{proposition} Let $k\in\{-q+1,\ldots,p-1\}$ be such that the
hypotheses H2$k$ and H3$k$ hold true. Then the statements
\eqref{main2}--\eqref{weakconvergence:isolated} in Theorem~\ref{theorem:eigTn}
remain true, provided that we replace everywhere $\Gamma_0$, $G_0$,
$\mu_{0,n}$, $\mu_0$, $m_{0}^{(l)}$ and $\mathrm{sp}$ by $\Gamma_k$, $G_k$,
$\mu_{k,n}$, $\mu_k$, $m_{k}^{(l)}$ and $\mathrm{sp}_k$ respectively.
\end{proposition}

Unfortunately the hypothesis H3$k$ is very delicate to handle, and we have been
unable to obtain sufficient conditions in the style of
Prop.~\ref{prop:twoconditions} for a reasonably large class of symbols $A(z)$.
For this reason, we will not discuss generalized eigenvalues any further in
this paper.

\section{A case study: scalar banded matrices}
\label{section:casestudy:scalarband}

In this section we specialize our results to the case where $T_n(A)$ is a
\emph{scalar} banded matrix with non-vanishing outer diagonals. More precisely,
we assume that $T_n(A)$ is a banded block Toeplitz matrix as in
\eqref{blockToeplitzmatrix}, that can be written in the scalar form
\begin{equation}
T_n(A) =
\begin{pmatrix} \label{blockToeplitzmatrix:scalar}
a_0^{(0)} & \ldots & a_0^{(-q)}  & & & 0 \\
\vdots & \ddots & & \ddots & &  \\
a_p^{(p)} &  & \ddots & & \ddots &  \\
 & \ddots & & \ddots &  & a_{rn-q-1}^{(-q)} \\
  & & \ddots &  & \ddots & \vdots \\
0   & & & a_{rn-1}^{(p)} & \ldots & a_{rn-1}^{(0)}
\end{pmatrix}_{rn\times rn},
\end{equation}
where the numbers $a_i^{(k)}\in\cee$ are such that
\begin{equation}\label{recurrent:as}
a_i^{(k)}=a_{i\textrm{ mod }r}^{(k)},\qquad
\end{equation} for all $i\in\enn\cup\{0\}$ and $k=-q,\ldots,p$, and with
\begin{equation}
\label{convention:nonzero} a_i^{(p)}\neq 0,\qquad a_i^{(-q)}\neq 0,
\end{equation}
for all $i=0,\ldots,r-1$. To avoid trivial cases we again assume that
$\min(p,q)\geq 1$. We will see in a moment that the notations $p$ and $q$ in
\eqref{blockToeplitzmatrix:scalar} are consistent with those used before in
\eqref{f:expansion}.

The representations \eqref{blockToeplitzmatrix} and
\eqref{blockToeplitzmatrix:scalar} are related as follows:
\begin{equation}\label{def:smod} \alpha:=\lceil q/r \rceil,\qquad \beta:=\lceil p/r \rceil,\end{equation}
and the matrices $A_k$, $k=-\alpha,\ldots,0$ in \eqref{blockToeplitzmatrix} are
obtained by taking the submatrix formed by the first $r$ rows of
\eqref{blockToeplitzmatrix:scalar} and partitioning it in blocks of size
$r\times r$ as follows:
\begin{equation}\label{def:Toeplitzblocks}
\begin{pmatrix} a_0^{(0)} & \ldots & a_0^{(-q)} &  & 0 & \\
\vdots &  & & \ddots \\
a_{r-1}^{(r-1)} & \ldots & \ldots & \ldots & a^{(-q)}_{r-1}
\end{pmatrix} = \begin{pmatrix}
A_0 & A_1 & \ldots & A_{-\alpha}
\end{pmatrix}.
\end{equation}
Here we add $r\alpha-q = r\lceil q/r \rceil -q$ zero columns at the right of
the matrix in the left hand side of \eqref{def:Toeplitzblocks} in order to have
compatible matrix dimensions. Similarly the matrices $A_k$, $k=0,\ldots,\beta$
are obtained by taking the submatrix formed by the first $r$ \emph{columns} of
\eqref{blockToeplitzmatrix:scalar} and partitioning it in blocks of size
$r\times r$.

One checks that the symbol $A(z)$ can be written as
\begin{equation}\label{symbol:shiftmx}
A(z) = \sum_{k=-q}^p \diag(a_0^{(k)},\ldots,a_{r-1}^{(k)})S^k,
\end{equation}
where $S$ is the cyclic shift matrix in \eqref{def:Sshiftmx}. There is also the
alternative representation
\begin{equation}\label{symbol:shiftmx:sym}
DA(z^r)D^{-1} = \sum_{k=-q}^p
z^k\diag(a_0^{(k)},\ldots,a_{r-1}^{(k)})\widetilde S^k,
\end{equation}
where $D:=\diag(1,z,\ldots,z^{r-1})$ and
\begin{equation}\label{def:Sshiftmx:bis} \widetilde S:=\begin{pmatrix} 0 & 1
\\ I_{r-1} & 0
\end{pmatrix}. \end{equation}
One may argue that \eqref{symbol:shiftmx:sym} is more natural than
\eqref{symbol:shiftmx}, in the sense that it gives the same weight $z^k$ to
\emph{all} the entries on the $k$th scalar diagonal of the matrix
\eqref{blockToeplitzmatrix:scalar}. From this representation we also obtain
that
\begin{equation}\label{def:algcurve:sym}
f(z^r,\lam) = \det\left(-\lam I_r+\sum_{k=-q}^p
z^k\diag(a_0^{(k)},\ldots,a_{r-1}^{(k)})\widetilde S^k\right),
\end{equation}
recall \eqref{def:f}.

\begin{proposition}\label{prop:fstructure} (Structure of $f$). Let $T_n(A)$ be as in
\eqref{blockToeplitzmatrix:scalar}--\eqref{convention:nonzero}. Then
$f(z,\lam)$ in \eqref{def:algcurve:sym} can be written in the form
\begin{equation}\label{f:series}
f(z,\lam) = f_{-q}(\lam)z^{-q}+\ldots + f_0(\lam) + \ldots+ f_{p}(\lam)z^p,
\end{equation}
where all the coefficients $f_k(\lam)$, $k=-q,\ldots,p$ are polynomials in
$\lam$. The outermost coefficients take the values
\begin{equation}\label{fp}
f_{-q}(\lam) \equiv f_{-q} = (-1)^{q(r-1)}\prod_{k=0}^{r-1}a_k^{(-q)},\qquad
f_p(\lam) \equiv f_p = (-1)^{p(r-1)}\prod_{k=0}^{r-1}a_k^{(p)},
\end{equation}
so hypothesis H1 holds true. For general $k$, the degree of the polynomial
$f_k(\lam)$ is bounded by
\begin{equation}\label{fk:degree1}
\deg f_k \leq \left\{\begin{array}{ll}\frac{q+k}{q}r,& \qquad \textrm{for }k=-q,\ldots,0,\\
\frac{p-k}{p}r,& \qquad \textrm{for }k=0,\ldots,p.\end{array}\right.
\end{equation}
\end{proposition}

\begin{proof}
Equations \eqref{f:series}--\eqref{fp} follow immediately from
\eqref{def:algcurve:sym}. To prove \eqref{fk:degree1} one can use a
combinatorial argument in the style of \cite[Proof of Prop.~2.5]{DK}; a simpler
proof will be obtained in Example~\ref{example:regularsj} below.
\end{proof}

\begin{corollary} Under the conditions of Prop.~\ref{prop:fstructure}, we
have that
\begin{equation}
\label{mk:Duits:bis} (m_k)_{k=-q}^p =(\til m_k)_{k=-q}^p= \left(0,\frac
rq,\frac{2r}{q},\ldots,r,\ldots,\frac{2r}{p} ,\frac rp,0\right).
\end{equation}
So the total masses $\til m_k$ of the measures $\mu_k$ form a simple arithmetic
series in the same way as in the scalar Toeplitz case, see
Example~\ref{example:mk1}. The energy functional of the equilibrium problem
reduces to \eqref{energyfunctional:H1}.
\end{corollary}

\section{Proofs}
\label{section:proofs}

In this section we prove our main results.

\subsection{Some preliminaries}

First we single out some preliminaries which will be repeatedly used in the
proofs.

\subsubsection*{Asymptotics of the roots $z_j(\lam)$}

Consider an algebraic curve $f(z,\lam)=0$ as in \eqref{f:expansion:arb} and
define the roots $z_j(\lam)$ as in \eqref{ordering:roots} and curves $\Gamma_k$
as in \eqref{def:Gammak:mod}. Let $j\in\{1,\ldots,p+q\}$ be fixed. It is
well-known that there exist constants $s_j\in\er$, $c_j\in\cee\setminus\{0\}$
and $\kappa_j\in\enn$ such that
\begin{equation}\label{def:sj}z_j(\lam) = c_j\lam^{s_j}\left(1+O\left(\lam^{-1/\kappa_j}\right)\right),\end{equation}
as $\lam\to\infty$ with $\lam\in\cee\setminus\bigcup_k \Gamma_k$, with possibly
a different constant $c_j$ for each connected component of
$\cee\setminus\bigcup_k \Gamma_k$ in which we let $\lam\to\infty$. Obviously,
\begin{equation}\label{ordering:sj}s_1\leq \ldots\leq s_{p+q},\end{equation} because of the ordering
\eqref{ordering:roots} of the roots $z_j(\lam)$.

The expansion \eqref{def:sj} is an instance of a \emph{Puiseux series} and the
next lemma is a well-known result for the Newton polygon. We include the proof
for completeness.

\begin{lemma}\label{lemma:sj} The numbers $s_j$ in \eqref{def:sj} are such that
\begin{equation}\label{Newton:polygon}
\sum_{j=1}^{q+k}s_j = \deg (f_{-q})-m_k,\end{equation} for any $k=-q,\ldots,p$,
with $m_k$ as in Definition~\ref{def:mk}.\end{lemma}

\begin{proof}
We start from the factorization
\begin{equation*}
f(z,\lam) = \frac{f_p(\lam)}{z^q}\prod_{j=1}^{p+q}(z-z_j(\lam)).
\end{equation*}
By expanding this product in powers of $z$, we see that the coefficient
$f_k(\lam)$ in \eqref{f:expansion:arb} is given by
\begin{equation}\label{def:ak} f_k(\lam) = (-1)^{p-k}f_p(\lam) \sum_{S}
\prod_{j\in S} z_j(\lam),
\end{equation}
where the summation runs over all subsets $S\subset\{1,\ldots,p+q\}$ with $|S|
= p-k$, for any $k\in\{-q,\ldots,p\}$. Then we obtain
\begin{equation}\label{estimate:fk}
f_k(\lam)/f_p(\lam) = O\left(\sum_{S}|\lam|^{\sum_{j\in S} s_j}\right) =
O\left(|\lam|^{\sum_{j\in S_k} s_j}\right),\qquad \lam\to\infty,
\end{equation}
for any $k\in\{-q,\ldots,p\}$, where in the second step we define
$S_k:=\{q+k+1,\ldots,p+q\}$. Hence
\begin{equation}\label{inequalityproof1} \sum_{j=q+k+1}^{p+q} s_j \geq
\deg f_k - \deg f_p.
\end{equation}
Moreover if $s_{q+k}<s_{q+k+1}$ then equality must hold in
\eqref{inequalityproof1}, since in that case $S = S_k$ yields the unique
dominant summand in the middle term of \eqref{estimate:fk}. In particular this
holds for $k=-q$:
\begin{equation}\label{sj:sumiszero} \sum_{j=1}^{p+q} s_j = \deg
f_{-q}-\deg f_p.
\end{equation}
By subtracting \eqref{sj:sumiszero} from \eqref{inequalityproof1} we then get
\begin{equation}\label{inequalityproof3}
\deg f_{-q}-\sum_{j=1}^{q+k}s_j \geq \deg f_k,
\end{equation}
with equality if $s_{q+k}<s_{q+k+1}$.

Denote by $\widehat m_k$ the left hand side of \eqref{inequalityproof3}. Then
$k\mapsto \widehat m_k$ is a concave function on $\{-q,\ldots,p\}$ by virtue of
\eqref{ordering:sj}. From \eqref{inequalityproof3} we see that $\widehat
m_k\geq \deg f_k$, with equality for each $k$ for which $2\widehat m_k>\widehat
m_{k-1}+\widehat m_{k+1}$, i.e., for each $k$ for which the concave, piecewise
linear function that interpolates between the grid points $(k,\widehat m_k)$
changes slope. Then Definition~\ref{def:mk} implies that $\widehat m_k=m_k$,
which is \eqref{Newton:polygon}.\end{proof}

\begin{corollary}\label{corollary:zjinf:brute}
Under the assumption \eqref{f:expansionlargelambda} we have that
\begin{equation}\label{zj:infty:brute}
\left\{\begin{array}{ll}  z_j(\lam)\to 0, &\quad j=1,\ldots,q,\\
z_j(\lam)\to\infty, &\quad j=q+1,\ldots,p+q,\end{array}\right.
\end{equation}
as $\lam\to\infty$ with $\lam\in\cee\setminus\bigcup_k \Gamma_k$. In
particular, the set $\Gamma_0$ in \eqref{def:Gammak:mod} (or
\eqref{def:Gammak}) is compact.
\end{corollary}

\begin{proof}
The assumption \eqref{f:expansionlargelambda} implies that $k\mapsto m_k$ is a
strictly increasing function on $\{-q,\ldots,0\}$ and strictly decreasing on
$\{0,\ldots,p\}$. Thus \eqref{Newton:polygon} implies that $s_j<0$ for
$j=1,\ldots,q$ and $s_j>0$ for $j=q+1,\ldots,p+q$. The result then follows from
\eqref{def:sj}.
\end{proof}

\begin{example}\label{example:regularsj} Let $f(z,\lam)=0$ be an algebraic curve as in
\eqref{def:algcurve:sym} and \eqref{convention:nonzero}. Then
\begin{equation}\label{sj:regular}
s_j=\left\{\begin{array}{ll}  -r/q, &\quad j=1,\ldots,q,\\
\ \ r/p, &\quad j=q+1,\ldots,p+q.\end{array}\right.
\end{equation}
Indeed, by virtue of \eqref{def:sj} and \eqref{def:algcurve:sym} we find that
$$0=f(z_j(\lam),\lam)=\left((-1)^r\lam^r+c_1\lam^{-qs_{j}}+c_2\lam^{ps_j}\right)(1+o(1)),\qquad \lam\to\infty,$$
for certain non-zero constants $c_1,c_2$. For this expression to be zero for
large $\lam$ we must have that two out of the three exponents
$\{r,-qs_j,ps_j\}$ are equal and the third is smaller; this implies that either
$s_j=-r/q$ or $s_j=r/p$, for all $j=1,\ldots,p+q$. The fact that $s_j=-r/q$
occurs with multiplicity $q$ and $s_j=r/p$ occurs with multiplicity $p$, is
then a consequence of the relation $\sum_{j=1}^{p+q} s_j=0$, recall
\eqref{sj:sumiszero} and \eqref{fp}. Finally, we note that \eqref{sj:regular}
and \eqref{Newton:polygon} imply \eqref{mk:Duits:bis}, which in turn leads to
\eqref{fk:degree1}.
\end{example}

Similarly to the above discussion, for any $l=1,\ldots,L$ there exist constants
$s_j^{(l)}\in\er$, $c_j^{(l)}\in\cee\setminus\{0\}$ and $\kappa_j^{(l)}\in\enn$
such that
\begin{equation}\label{def:sj:l}z_j(\lam) = c_j^{(l)}(\lam-\lam_l)^{s_j^{(l)}}
+O\left((\lam-\lam_l)^{s_j^{(l)}+1/\kappa_j^{(l)}}\right),\end{equation} as
$\lam\to\lam_l$ with $\lam\in\cee\setminus\bigcup_k \Gamma_k$, with possibly a
different value of $c_j^{(l)}$ for each connected component of
$\cee\setminus\bigcup_k \Gamma_k$ in which we let $\lam\to\lam_l$. The numbers
$s_j^{(l)}$ are such that
\begin{equation}\label{Newton:polygon:l}
\sum_{j=1}^{q+k}s_j^{(l)} = \mult_{\lam-\lam_l}
(f_{-q})-m_k^{(l)},\end{equation} for any $k=-q,\ldots,p$ and $l=1,\ldots,L$.

\subsubsection*{Widom's determinant identity}

\begin{proposition}\label{proposition:Widomformula} (Widom's determinant identity). Let
$\lam\in\cee$ be such that the solutions $z_j(\lam)$ of the algebraic equation
$f(z,\lam)=0$ in \eqref{def:f} are pairwise distinct. Then for all $n$
sufficiently large we have
\begin{equation}\label{Widom:1} \det
T_n(A(z)-\lam I_r) = \sum_{S} C_{S}(\lam)(w_S(\lam))^{n+\alpha},
\end{equation}
where the sum is over all subsets $S\subset\{1,2,\ldots,p+q\}$ of cardinality
$|S| = q$ and for each such $S$ we have
\begin{equation}\label{def:wS} w_S(\lam) =
(-1)^q f_{-q}(\lam)\prod_{j\in S}z_j(\lam)^{-1}
\end{equation}
and
\begin{equation}\label{def:CS}
C_S(\lam) = \det\left( \frac{1}{2\pi i}\int_{\sigma_S} z^{\mu-\nu} (A(z)-\lam
I_r)^{-1}\frac{dz}{z} \right)_{\mu,\nu=1,\ldots,\alpha}
\end{equation}
where $\sigma_S$ is a counterclockwise oriented closed Jordan curve enclosing
$z=0$ and the points $z_j(\lam)$, $j\in S$, but no other roots of
$f(z,\lam)=0$.
\end{proposition}

Note that \eqref{def:wS} can be written alternatively as
$$ w_S(\lam) = (-1)^p
f_{p}(\lam)\prod_{j\in \overline{S}}z_j(\lam),$$  with
$\overline{S}:=\{1,2,\ldots,p+q\}\setminus S$. This expression has maximal
modulus among the subsets $S$ of cardinality $|S|=q$ if $S=S_0$ with
\begin{equation}\label{def:S0}
S_0 = \{1,\ldots,q\}.
\end{equation}
For $S=S_0$ the definition of $C_{S_0}$ in \eqref{def:CS} reduces to the one of
$C_0$ in \eqref{def:C0}.

Prop.~\ref{proposition:Widomformula} was obtained in \cite[Section 6]{Widom1}
by means of the Baxter-Schmidt formula \cite{BaxterSchmidt}. Note that
Prop.~\ref{proposition:Widomformula} assumes that $n$ is sufficiently large,
say $n\geq n_0$, but this is no problem since \cite[Section 6, Remark
1]{Widom1} guarantees that the same value of $n_0$ works for all~$\lam$.

Prop.~\ref{proposition:Widomformula} assumes that the solutions of the
algebraic equation $f(z,\lam)=0$ are pairwise distinct. If this assumption
fails then similar determinant formulas can be obtained, by taking a suitable
limit of \eqref{Widom:1} and using continuity. This will be hinted at in
Section~\ref{subsection:prooftheorem:eigTn}.

For the scalar case $r=1$ it is known that
$$ C_S(\lam) = \prod_{j\in \overline{S}}z_j(\lam)^{q}\prod_{j\in
\overline{S},l\in S}(z_j(\lam)-z_l(\lam))^{-1},
$$
and then Prop.~\ref{proposition:Widomformula} reduces to a result in
\cite{Widom0}.

\subsection{Proof of Proposition~\ref{prop:twoconditions}}
\label{subsection:proof:proptwo}

\begin{proof}[Proof of Proposition~\ref{prop:twoconditions}(a)]
Suppose that $\cee\setminus\Gamma_0$ is connected and moreover $\Gamma_0$ does
not have any interior points, recalling \eqref{def:Gammak}. From
Lemma~\ref{lemma:structureGammak} we immediately obtain H2. Next we establish
H3. Eq.~\eqref{zj:infty:brute} implies that for $|\lam|$ large enough we can
take the contour $\sigma_0$ in \eqref{def:C0} to be the unit circle. Then we
easily find that $$\left( \frac{1}{2\pi i}\int_{\sigma_0} z^{\mu-\nu}
(A(z)-\lam I_r)^{-1}\frac{dz}{z} \right)_{\mu,\nu=1,\ldots,\alpha}=-\lam^{-1}
I_{r\alpha}(1+O(\lam^{-1})),\qquad \lam\to\infty,$$ and therefore
$C_0(\lam)=(-\lam^{-1})^{r\alpha}(1+O(\lam^{-1}))\neq 0$ as $\lam\to\infty$.
Hypothesis H3 then follows from the analyticity of $C_0(\lam)$ in
$\cee\setminus\Gamma_0$.
\end{proof}

\begin{proof}[Proof of Proposition~\ref{prop:twoconditions}(b)]
Assume that $T_n(A)$ is a Hessenberg matrix. Hence by definition, $T_n(A)$ has
the form \eqref{blockToeplitzmatrix:scalar} with $q=1$ and with superdiagonal
entries $a_i^{(-1)}\neq 0$ for all $i$. We will need some auxiliary lemmas.

By virtue of \eqref{f:series} (where now $q=1$) we see that the equation
$f(z,\lam)=0$ has $p+1$ roots
\begin{equation}\label{def:zj}z_j=z_j(\lam),\qquad j = 1,\ldots,p+1,\end{equation}
for a certain $p\in\enn$ (which is not necessarily the same $p$ as in
\eqref{blockToeplitzmatrix:scalar}). Basic algebraic geometry shows that one
can choose a finite union of analytic arcs $\Gamma\subset\cee$ so that the
roots $z_j(\lam)$, $j = 1,\ldots,p+1$ depend analytically on $\lam\in\cee$,
except when $\lam\in \Gamma$. We will see in a moment that we can define
$\Gamma$ by means of \eqref{def:Gammak}, and the $z_j(\lam)$ as in
\eqref{ordering:roots}; but we are \emph{not} making these assumptions yet.

We consider the \emph{Riemann surface} $\mathcal R$ associated to the algebraic
equation $f(z,\lam)=0$: it is a branched $(p+1)$-sheeted covering of
$\overline\cee$, with the analytic function $z_j(\lam)$ defined for $\lam$ on
the $j$th sheet $\overline\cee$, $j=1,\ldots,p+1$. These functions have a cut
along the appropriate arcs of $\Gamma$, and the different sheets of $\mathcal
R$ are glued together along these arcs.

\begin{lemma}\label{lemma:Rsurface} 
Let the roots $z_j(\lam)$ in \eqref{def:zj} and the Riemann surface $\mathcal
R$ be defined as in the previous paragraph. Then $\mathcal R$ is connected.
\end{lemma}

\begin{proof} 
Take an arbitrary point $(\lam_0,z_0)\in\mathcal R$. Define the set
$$ \mathcal Z:=\{z\in\overline\cee\mid
\textrm{there exists $\lam\in\overline{\cee}$ and a continuous path in
$\mathcal R$ from $(\lam_0,z_0)$ to $(\lam,z)$}\}.
$$
Then $\mathcal Z$ is a subset of $\overline\cee$ which is both open and closed.
Hence it must be the entire Riemann sphere $\overline\cee$. In particular it
contains the value $z=0$. But to $z=0$ there corresponds only $\lam=\infty$
(use \eqref{f:series}--\eqref{fp} with $q=1$). 
Moreover, there is a unique such point $(\lam,z)=(\infty,0)$ on the Riemann
surface (use \eqref{zj:infty:brute} with $q=1$). Summarizing, we see that there
is a continuous path in $\mathcal R$ from $(\lam_0,z_0)$ to this unique
reference point $(\infty,0)$. Since this holds true for any
$(\lam_0,z_0)\in\mathcal R$, the connectedness of $\mathcal R$ follows.
\end{proof}

From now on we will order the roots $z_j=z_j(\lam)$, $j=1,\ldots,p+1$, by
increasing modulus as in \eqref{ordering:roots}. We also define the sets
$\Gamma_k$, $k=0,\ldots,p-1$, as in \eqref{def:Gammak}.

\begin{lemma}\label{lemma:Gammak1D} Each set $\Gamma_k$, $k=0,\ldots,p-1$ is a finite union of analytic
arcs and points in $\cee$. Hence, the sets $\Gamma_k$ can be taken as cuts for
the Riemann surface $\mathcal R$.
\end{lemma}

\begin{proof} The proof boils down to showing that $\Gamma_k$ does not contain
a (two-dimensional) open disk $U\subset\cee\setminus\Lam$. In that case we
would have two roots $z_i(\lam)$ and $z_{j}(\lam)$ that have identically equal
modulus in $U$. Their ratio must be a constant of modulus one. We then obtain a
contradiction by using the connectedness of the Riemann surface associated to
$f(z,\lam)=0$ (Lemma~\ref{lemma:Rsurface}), and the fact that there is only one
root $z_1(\lam)$ that goes to zero if $\lam$ goes to $\infty$ (see
\eqref{zj:infty:brute} with $q=1$).
\end{proof}

Since $q=1$, \eqref{Widom:1} now specializes to the form
\begin{equation}\label{Widom:2} \det
(T_n(A)-\lam I_{rn}) = \sum_{j=1}^{p+1} C_{j-1}(\lam)(-z_j(\lam))^{-n-1},
\end{equation}
with
\begin{equation}\label{def:CS:hess}
C_{j-1}(\lam) := \det\left( \frac{1}{2\pi i}\int_{\sigma} (A(z)-\lam
I_r)^{-1}\frac{dz}{z} \right),
\end{equation}
where $\sigma$ is a counterclockwise oriented closed Jordan curve enclosing
$z=0$ and the point $z_j(\lam)$, but none of the other roots $z_i(\lam)$,
$i\in\{1,\ldots,p+1\}$, $i\neq j$. Here we write $C_{j-1}$ rather than $C_j$ to
be consistent with \eqref{def:C0}. The function $C_{j-1}(\lam)$ is defined for
$\lam$ in the domain
$\mathcal D_j := \cee\setminus(\Gamma_{j-1}\cup\Gamma_j).$

\begin{lemma} 
\label{lemma:widom:explicit} The function $C_{j-1}(\lam)$ in
\eqref{def:CS:hess}, $j\in\{1,\ldots,p+1\}$, has only isolated zeros in
$\mathcal D_j$.
\end{lemma}

\begin{proof}
We must show that $C_{j-1}(\lam)$ cannot be identically zero in any open disk
in $\cee\setminus\bigcup_k\Gamma_k$. By the fact that the Riemann surface is
connected (Lemma~\ref{lemma:Rsurface}) and analytic continuation, this would
imply that \emph{each} of the $C_{j-1}(\lam)$, $j=1,\ldots,p+1$, is identically
zero in $\mathcal D_j$. But then \eqref{Widom:2} would imply that
$\det(T_n(A)-\lam I_{rn})\equiv 0$ which is clearly a contradiction.
\end{proof}

Combining the above two lemmas, we have now established that H2$k$ and H3 hold
true when $T_n(A)$ has Hessenberg structure. This ends the proof of
Proposition~\ref{prop:twoconditions}(b). \end{proof}

\subsection{Proofs of Lemmas~\ref{lemma:mkpos} and \ref{lemma:equality}}
\label{subsection:proof:lemmaremark}

\begin{proof}[Proof of Lemma~\ref{lemma:mkpos}] From the definition of $m_k$
we trivially have that \begin{equation}\label{proof:convex1} m_k \geq
\frac{p-k}{p+q}\deg f_{-q}+\frac{k+q}{p+q}\deg f_p,
\end{equation}
while from the definition of $m_k^{(l)}$ it follows that
\begin{equation}\label{proof:convex2} m_k^{(l)} \leq
\frac{p-k}{p+q}\mult_{\lam-\lam_l} f_{-q}+ \frac{k+q}{p+q}\mult_{\lam-\lam_l}
f_p,
\end{equation}
for any $l=1,\ldots,L$. Summing \eqref{proof:convex2} for all $l=1,\ldots,L$
and subtracting this from \eqref{proof:convex1}, we then obtain the desired
inequality \eqref{totalsumofmasses:pos} upon using that
\begin{equation}\label{proof:convex3} \deg f_{-q}=\sum_{l=1}^L
\mult_{\lam-\lam_l} f_{-q},\qquad \deg f_{p}=\sum_{l=1}^L \mult_{\lam-\lam_l}
f_{p}.
\end{equation}

Next we check the statement about the strictness of the inequality
\eqref{totalsumofmasses:pos}. From the arguments in the above paragraph we see
that equality in \eqref{totalsumofmasses:pos} can be achieved only if equality
holds in both \eqref{proof:convex1} and \eqref{proof:convex2}. For
\eqref{proof:convex1}, this means graphically that the grid point $(k,m_k)$
lies on the line segment connecting $(-q,\deg f_{-q})$ and $(p,\deg f_{p})$.
From the definition of the numbers $m_k$ it then follows that each of the grid
points $(k,\deg f_k)$, $k=-q,\ldots,p$, must lie below this line segment, in
the sense that
\begin{equation}\label{proof:convex3bis}
\deg f_k\leq \frac{p-k}{p+q}\deg f_{-q}+\frac{k+q}{p+q}\deg f_p,
\end{equation}
for all $k=-q,\ldots,p$. Similarly, equality in \eqref{proof:convex2} implies
that
\begin{equation}\label{proof:convex4bis}
\mult_{\lam-\lam_l} f_k \geq \frac{p-k}{p+q}\mult_{\lam-\lam_l} f_{-q}+
\frac{k+q}{p+q}\mult_{\lam-\lam_l} f_p,\end{equation} for all $k=-q,\ldots,p$
and $l=1,\ldots,L$.

Now if \eqref{f:expansionlargelambda} holds then we have $\deg f_0 = r$ while
the right hand side of \eqref{proof:convex3bis} is at most $r-1$. So we cannot
have $\til m_k=0$ in that case.
\end{proof}

\begin{proof}[Proof of Lemma~\ref{lemma:equality}]
First we show that (i) implies (iii). So suppose that $\til m_k=0$ for some
$k\in\{-q+1,\ldots,p-1\}$. As observed before, we then have the inequalities
\eqref{proof:convex3bis}--\eqref{proof:convex4bis} for all $k=-q,\ldots,p$.
Summing \eqref{proof:convex4bis} for all $l$ and subtracting this from
\eqref{proof:convex3bis}, we get
\begin{equation}\label{proof:convex4} \deg f_k - \sum_{l=1}^L
\mult_{\lam-\lam_l} f_k \leq 0,
\end{equation}
for any $k=-q,\ldots,p$, where the right hand side was simplified with the help
of \eqref{proof:convex3}. On the other hand, we trivially have that
$$ \deg f_k - \sum_{l=1}^L \mult_{\lam-\lam_l} f_k \geq 0.
$$
So equality holds in \eqref{proof:convex4}. Tracing back the argument, we must
then have equality in both \eqref{proof:convex3bis} and
\eqref{proof:convex4bis}. Graphically this means that each of the grid points
$(k,\deg f_k)$, $k=-q,\ldots,p$ must lie on the line segment connecting
$(-q,\deg f_{-q})$ and $(p,\deg f_{p})$, and similarly each of the grid points
$(k,\mult_{\lam-\lam_l} f_k)$, $k=-q,\ldots,p$ must lie on the line segment
connecting $(-q,\mult_{\lam-\lam_l} f_{-q})$ and $(p,\mult_{\lam-\lam_l}
f_{p})$, for any $l=1,\ldots,L$. It is easily seen that these assertions are
equivalent to the statement in part (iii), with the rational function $h(\lam)$
given by
$$ h(\lam) := \prod_{l=1}^L (\lam-\lam_l)^{(\mult_{\lam-\lam_l}
(f_p/f_{-q}))/(p+q)}.
$$
So we showed that (i) implies (iii). The proof that (iii) implies (ii) can be
obtained (in a simpler way) by reversing the above arguments.
\end{proof}

\subsection{Proof of Lemma~\ref{lemma:strucwk}}
\label{subsection:proof:lemmamerom}

The proof of (a) follows again by mimicking the argument of Schmidt and Spitzer
\cite{SchmidtSpitzer}. For Part~(b), let
$U\subset\cee\setminus(\Gamma_k\cup\Lam)$ be a simply connected domain. For
fixed $\lam\in U$ let $j_1,j_2\in\enn\cup\{0\}$ be such that
$$|z_{q+k-j_1-1}(\lam)|<|z_{q+k-j_1}(\lam)| = \ldots =
|z_{q+k+j_2}(\lam)|<|z_{q+k+j_2+1}(\lam)|,$$ where we set $z_0\equiv 0$ and
$z_{p+q+1}\equiv \infty$ if necessary. Since $U\subset\cee\setminus\Gamma_k$ we
have that either $j_2\equiv 0$ on $U$, or else $j_1$ and $j_2$ both take a
constant value on $U$. In the case where $j_2\equiv 0$ the analyticity of
$\prod_{j=1}^{q+k} z_j(\lam)$ on $U$ follows immediately from \cite[Proof of
Prop.~3.5]{DK}. So we can focus now on the case where $j_1$ and $j_2$ are
constant on $U$. Since $U\subset\cee\setminus\Lam$, none of the roots $z_j$,
$j=q+k-j_1,\ldots,q+k+j_2$ can take the value $0$ or $\infty$. Hence by the
fact that $U$ is simply connected, there exists a labeling so that each of the
functions $z_{q+k-j_1}(\lam),\ldots,z_{q+k+j_2}(\lam)$ is analytic in $U$, with
the pairwise ratios being constants of modulus $1$. On the other hand, the
argument in \cite[Proof of Prop.~3.5]{DK} shows that $\prod_{j=1}^{q+k-j_1-1}
z_j(\lam)$ is analytic in $U$. Combining all these observations, we obtain the
required analyticity of $\prod_{j=1}^{q+k} z_j(\lam)$ in $U$.

Finally, the statement about the logarithmic derivative follows since if
$z(\lam)$ and $\tilde z(\lam)$ are analytic functions of $\lam\in U$ that are
identically equal up to a constant factor, then their logarithmic derivatives
are the same: $z'(\lam)/z(\lam) = \tilde z'(\lam)/\tilde z(\lam),$ for all
$\lam\in U.$ $\bol$

\subsection{Proofs of Proposition~\ref{prop:componentsGammak} and Theorems \ref{theorem:main1} and \ref{theorem:main1:mod}}
\label{subsection:proof:theorem:main1}

In this section we prove Prop.~\ref{prop:componentsGammak} and Theorems
\ref{theorem:main1} and \ref{theorem:main1:mod}. The proof of
Theorem~\ref{theorem:main1} will closely follow \cite{DK} and
especially~\cite{DD}.

%
Define the function $w_k$ by
\begin{equation}\label{def:wk}
w_k(\lam) = \prod_{j=1}^{q+k} z_j(\lam),\quad \lam\in\cee\setminus\Gamma_k,
\end{equation}
for $k=-q+1,\ldots,p-1$.
Occasionally we will also consider $w_k$ for the indices $k=-q$ or $k=p$.

We rewrite \eqref{def:measuresk} as \begin{equation}\label{def:measuresk:bis}
d\mu_k(\lam) = \frac{1}{2\pi i}\left(
\frac{w_{k+}'(\lam)}{w_{k+}(\lam)}-\frac{w_{k-}'(\lam)}{w_{k-}(\lam)}
\right)d\lam.
\end{equation}
From Lemma~\ref{lemma:strucwk} we know that $w_{k}'/w_{k}$ exists as a
meromorphic function on $\cee\setminus\Gamma_k$ with poles at the points of
$\Lam$. The following proposition gives more detailed information.

\begin{proposition}\label{prop:threeproperties} (The function $w_{k}'/w_{k}$).
Let $k\in\{-q+1,\ldots,p-1\}$ and recall Definitions \ref{def:mk} and
\ref{def:mkl}. Then the following statements hold true:
\begin{enumerate}
\item[(a)] For any $\lam_0\in\cee\setminus\Lam$,
there exists $\kappa\in\enn=\{1,2,3,\ldots\}$ such that
$$
\frac{w_k'(\lam)}{w_k(\lam)} = O((\lam-\lam_0)^{-1+1/\kappa}),
$$
as $\lam\to\lam_0$ with $\lam\in\cee\setminus\Gamma_k$. We have $\kappa=1$ for
all but finitely many $\lam_0$.

\item[(b)] Near $\infty$ there exists $\kappa\in\enn$ such that
$$ \frac{w_k'(\lam)}{w_k(\lam)} = \frac{\deg (f_{-q})-m_k}{\lam}+
O(\lam^{-1-1/\kappa}),
$$
as $\lam\to\infty$ with $\lam\in\cee\setminus\Gamma_k$.

\item[(c)] Near the point $\lam_l$, $l\in\{1,\ldots,L\}$, there exists $\kappa\in\enn$ such that
$$
\frac{w_k'(\lam)}{w_k(\lam)} =
\frac{\mult_{\lam-\lam_l}(f_{-q})-m_{k}^{(l)}}{\lam-\lam_l}
 + O((\lam-\lam_l)^{-1+1/\kappa}),
$$
as $\lam\to\lam_l$ with $\lam\in\cee\setminus\Gamma_k$.
\end{enumerate}
\end{proposition}

\begin{proof}
Part (a) can be shown as in \cite{DD}, for example. Now we turn to proving Part
(b). Recalling the notation $s_j$ in \eqref{def:sj}, we obtain
\begin{equation*} \frac{w_k'(\lam)}{w_k(\lam)} =
\sum_{j=1}^{q+k}\frac{z_j'(\lam)}{z_j(\lam)} = \frac{\sum_{j=1}^{q+k}
s_j}{\lam} + O\left(\lam^{-1-1/\kappa}\right),
\end{equation*}
as $\lam\to\infty$ with $\lam\in\cee\setminus\Gamma_k$. On account of
Lemma~\ref{lemma:sj} we then obtain part (b). The proof of part (c) follows in
a similar way from \eqref{def:sj:l}--\eqref{Newton:polygon:l}.
\end{proof}

\begin{proposition}\label{prop:mupositivemass}
For each $k=-q+1,\ldots,p-1$ we have that $\mu_k$ in \eqref{def:measuresk:bis}
is a positive measure on $\wtil\Gamma_k$ with total mass $\mu_k(\wtil\Gamma_k)
= \til m_k$.
\end{proposition}

\begin{proof} (Compare with \cite[Prop~3.4]{DD}.) Prop.~\ref{prop:threeproperties} implies that the
density \eqref{def:measuresk:bis} is locally integrable around all the points
in $(\Lam\cup\{\infty\})\cap\wtil\Gamma_k$, and the arguments in \cite{DK} show
that $\mu_k$ is a positive measure. The statement that $\mu_k(\wtil\Gamma_k) =
\til m_k$ follows from the contour deformation
\begin{align}\nonumber \mu_k(\wtil\Gamma_k) &:= \frac{1}{2\pi i}\int_{\wtil\Gamma_k}\left(
\frac{w_{k+}'(\lam)}{w_{k+}(\lam)}-\frac{w_{k-}'(\lam)}{w_{k-}(\lam)}
\right)d\lam \\
\label{residues1}&= \frac{1}{2\pi i}\int_{\mathcal C}
\frac{w_{k}'(\lam)}{w_{k}(\lam)}\ d\lam +\sum_{l=1}^L
\res\left(\frac{w_{k}'(\lam)}{w_{k}(\lam)},\lam=\lam_l\right),
\end{align}
where $\mathcal C$ is a clockwise oriented contour surrounding
$\wtil\Gamma_k\cup\Lam$, and where $\res(h,\lam)$ denotes the residue of $h$ at
$\lam$. Equation \eqref{residues1} is valid even if one or more points
$\lam_l\in\Lam$ lie on the curve $\wtil\Gamma_k$, thanks to the local
integrability of $\mu_k$ around these points. Applying the residue theorem once
again, now for the exterior domain of $\mathcal C$, we find for the first term
in \eqref{residues1} that
\begin{equation}\label{residues2} \frac{1}{2\pi i}\int_{\mathcal C}
\frac{w_{k}'(\lam)}{w_{k}(\lam)}\ d\lam =
-\res\left(\frac{w_{k}'(\lam)}{w_{k}(\lam)},\lam=\infty\right).
\end{equation}
From \eqref{residues1}--\eqref{residues2} and the residue expressions in
Prop.~\ref{prop:threeproperties} we then obtain $$\mu_k(\wtil\Gamma_k) =
\left(m_k-\sum_{l=1}^L m_k^{(l)}\right)-\left(\deg f_{-q}-\sum_{l=1}^L
\mult_{\lam-\lam_l} f_{-q}\right) = \til m_k,$$ by virtue of
\eqref{totalsumofmasses} and \eqref{proof:convex3}.
\end{proof}

\begin{proposition}\label{prop:cauchytransformcontour}
For each $k$ we have that \begin{equation}\label{cauchytransform:hulp}
\int\frac{d\mu_k(x)}{\lam-x} = -\frac{w_k'(\lam)}{w_k(\lam)}+\sum_{l=1}^L
\frac{\mult_{\lam-\lam_l} (f_{-q})-m_{k}^{(l)}}{\lam-\lam_l},\quad\textrm{if
}\lam\in\cee\setminus\wtil\Gamma_k\end{equation} and
\begin{multline}\label{cauchytransform:muk} \int\log|\lam-x|\
d\mu_k(x) = -\log|w_k(\lam)|+\sum_{l=1}^L (\mult_{\lam-\lam_l}
(f_{-q})-m_{k}^{(l)})\log|\lam-\lam_l|+\alpha_k,
\end{multline}
if $\lam\in\cee$, for a suitable constant $\alpha_k$.
\end{proposition}

\begin{remark} As in \cite{DD}, each
$\lam_l\in\Lam\setminus\wtil\Gamma_k$ (or $\lam_l\in\Lam$) is a removable
singularity for the right hand side of \eqref{cauchytransform:hulp} (or
\eqref{cauchytransform:muk} respectively) due to the continuity of the
corresponding left hand side.
\end{remark}

\begin{proof}[Proof of Proposition~\ref{prop:cauchytransformcontour}]
(Compare with \cite[Prop~3.5]{DD}.) We use the contour deformation
\begin{eqnarray*}\int_{\wtil\Gamma_k}\frac{d\mu_k(x)}{\lam-x} &:=& \frac{1}{2\pi i}\int_{\wtil\Gamma_k}\frac{1}{\lam-x}
\left(\frac{w_{k+}'(x)}{w_{k+}(x)}-\frac{w_{k-}'(x)}{w_{k-}(x)} \right)dx
\\ &=& \frac{1}{2\pi i}\int_{\mathcal C}
\frac{1}{\lam-x}\frac{w_{k}'(x)}{w_{k}(x)}\ dx
-\frac{w_k'(\lam)}{w_k(\lam)} + 
\sum_{l=1}^L
\frac{1}{\lam-\lam_l}\res\left(\frac{w_{k}'(x)}{w_{k}(x)},x=\lam_l\right),
\end{eqnarray*}
where $\mathcal C$ is a clockwise oriented contour surrounding
$\wtil\Gamma_k\cup\Lam\cup\{\lam\}$. The first term in the right hand side
vanishes since the residue of the integrand at infinity is zero. From the
residue expressions in Prop.~\ref{prop:threeproperties} we then get
\eqref{cauchytransform:hulp}.
Eq.~\eqref{cauchytransform:muk} follows from this by integration, see also
\cite{DK}.
\end{proof}

Finally, we can finish the proofs of Theorems~\ref{theorem:main1} and
\ref{theorem:main1:mod}:

\begin{proof}[Proof of Theorem \ref{theorem:main1}(a)--(b)] With Propositions
\ref{prop:mupositivemass} and \ref{prop:cauchytransformcontour} in place,
Theorem \ref{theorem:main1}(a)--(b) now follows in the same way as in
\cite{DD}.
\end{proof}

\begin{proof}[Proof of Theorem \ref{theorem:main1}(c)] First we show that the energy
functional $J(\vec{\nu})$ in \eqref{energyfunctional} is bounded from below. To
this end we rewrite $J(\vec{\nu})$ as
\begin{multline}\label{alternativereprE} J(\vec{\nu}) = \sum_{k=-q+1}^{p-2} \frac{\til m_k \til m_{k+1}}{2}
I\!\left(\frac{\nu_k}{\til m_k} - \frac{\nu_{k+1}}{\til m_{k+1}}\right)
+\sum_{k=-q+1}^{p-1} \frac{\til m_k}{2}(- m_{k-1}+2 m_k- m_{k+1})
I\!\left(\frac{\nu_{k}}{\til m_k}\right)\\
+\sum_{l=1}^L\sum_{k=-q+1}^{p-1} \frac{\til
m_k}{2}(m_{k-1}^{(l)}-2m_k^{(l)}+m_{k+1}^{(l)})
\left(I\!\left(\frac{\nu_{k}}{\til m_k}\right)
-2\int\log\frac{1}{|\lam-\lam_{l}|}\ \frac{d\nu_k(\lam)}{\til
m_k}\right).\end{multline} This formula is easily shown with the help of
\eqref{energyfunctional} and \eqref{totalsumofmasses}.

The terms in the first sum in \eqref{alternativereprE} are all nonnegative
\cite{Simeonov}. For the second sum in \eqref{alternativereprE}, we observe
that $-m_{k-1}+2m_k-m_{k+1}  = s_{q+k+1}-s_{q+k}>0$ by virtue of
\eqref{Newton:polygon}. So these coefficients are all nonnegative, and they are
non-zero precisely when $s_{q+k+1}>s_{q+k}$. But for such $k$ the curve
$\wtil\Gamma_k$ is compact and so $I(\nu_k)$ is bounded from below. Finally,
for the double sum in \eqref{alternativereprE} we have that
$m_{k-1}^{(l)}-2m_k^{(l)}+m_{k+1}^{(l)}=s_{q+k}^{(l)}-s_{q+k+1}^{(l)}\geq 0$,
recall \eqref{def:sj:l}--\eqref{Newton:polygon:l}. So these coefficients are
all nonnegative, and they are non-zero precisely when
$s_{q+k}^{(l)}>s_{q+k+1}^{(l)}$. But for such $k$ we have that
$\lam_l\not\in\wtil\Gamma_k$, and then standard arguments from potential theory
show that the expression between brackets in the double sum in
\eqref{alternativereprE} is minimized precisely when $\nu_{k}$ is (a constant
times) the \emph{balayage} of the Dirac point mass at $\lam_l$ onto the curve
$\wtil\Gamma_{k}$; in particular this expression is bounded from below as well
\cite[Chapter 2]{SaffTotik}.

Summarizing, we have now established that the energy functional $J(\vec{\nu})$
is bounded from below. Then the proof of Theorem \ref{theorem:main1}(c) follows
from part (b) in the same way as in \cite{DD}.
\end{proof}

\begin{proof}[Proof of Proposition~\ref{prop:componentsGammak}]
Let us first prove \eqref{componentsGammak} if $\Lam\cap C=\emptyset$. Applying
contour deformation, we then find that
$$ \mu_k(C) := \frac{1}{2\pi i}\int_C \left(
\frac{w_{k+}'(\lam)}{w_{k+}(\lam)}-\frac{w_{k-}'(\lam)}{w_{k-}(\lam)}
\right)d\lam = \frac{1}{2\pi i}\int_{\gamma}
\frac{w_{k}'(\lam)}{w_{k}(\lam)}d\lam,
$$
where $\gamma$ is the disjoint union of one or more closed Jordan curves in
$\cee\setminus\Gamma_k$. More precisely, $\gamma$ consists of a clockwise
oriented loop surrounding the outer boundary of $C$, and a counterclockwise
oriented loop inside each of the \lq holes\rq\ of $C$. Now since
$\frac{w_{k}'(\lam)}{w_{k}(\lam)} = \left(\log w_{k}(\lam)\right)'$, the
integral of this quantity over any closed Jordan curve in
$\cee\setminus\Gamma_k$ is obviously an integral multiple of $2\pi i$. So we
obtain \eqref{componentsGammak} if $\Lam\cap C=\emptyset$. The same argument
also works if $\Lam\cap C\neq\emptyset$ provided that we take into account the
residue from the pole of $w_k'/w_k$ at each $\lam_l\in\Lam\cap C$,
Prop.~\ref{prop:threeproperties}(c), thereby noting that
$\mult_{\lam-\lam_l}f_{-q}\in\enn\cup\{0\}\subset\zet$.

Finally, denote by $K$ the number of compact, connected components of
$\Gamma_k$. By summing \eqref{componentsGammak} over all such components $C$ we
get the following upper bound on $K$:
$$ K\leq \sum_C \left(\mu_k(C)+\sum_{\lam_l\in \Lam\cap C} m_k^{(l)}\right)\leq \til m_k+\sum_{l=1}^L m_k^{(l)}\leq m_k,
$$
by virtue of \eqref{totalsumofmasses}.
\end{proof}

\begin{proof}[Proof of Theorem \ref{theorem:main1:mod}]
Thanks to Lemma~\ref{lemma:strucwk}, the above proof of
Theorem~\ref{theorem:main1} yields Theorem~\ref{theorem:main1:mod} as well.
\end{proof}

\subsection{Proof of Theorem \ref{theorem:eigTn} and Proposition~\ref{prop:ptmass:finiten}}
\label{subsection:prooftheorem:eigTn}

\begin{proof}[Proof of Prop.~\ref{prop:ptmass:finiten}]
Fix $\lam_l\in\Lam$ and a subset $S\subset\{1,2,\ldots,p+q\}$ of cardinality
$|S| = q$. From \eqref{def:wS} we have that
\begin{equation}\label{proof:zeroorder}w_S(\lam) = 
O\!\left((\lam-\lam_l)^{\mult_{\lam-\lam_l} (f_{-q})-\sum_{j=1}^q
s_k^{(l)}}\right)=O\!\left((\lam-\lam_l)^{m_0^{(l)}}\right),\end{equation} as
$\lam\to\lam_l$ with $\lam\in\cee\setminus\bigcup_k \Gamma_k$, where the last
step follows from \eqref{Newton:polygon:l}. Prop.~\ref{prop:ptmass:finiten}(a)
then follows from \eqref{proof:zeroorder} and \eqref{Widom:1},  provided that
there is a disk $U$ around $\lam_l$ such that for all but finitely many
$\lam\in U$ the roots to $f(z,\lam)=0$ are pairwise different. But this
condition is generic and the case where it fails follows by an easy continuity
argument.

The proof of Prop.~\ref{prop:ptmass:finiten}(b) is similar.
\end{proof}

Theorem \ref{theorem:eigTn} can be obtained from Widom's determinant identity,
Prop.~\ref{proposition:Widomformula}, in the same way as in \cite{DD,DK}. We
outline the main steps.

\begin{proposition}\label{prop:eig:aux}
We have that
\begin{equation}\label{cauchytransform:eig}
\lim_{n\to\infty} \int_{\cee} \frac{d\mu_{0,n}(x)}{\lam-x}= \int_{\cee}
\frac{d\mu_{0}(x)}{\lam-x} + \sum_{l=1}^L \frac{m_{0}^{(l)}}{\lam-\lam_l},
\end{equation}
uniformly on compact subsets of $\cee\setminus(\Gamma_0\cup G_0)$.
\end{proposition}

\begin{remark}
As in \cite{DD}, each $\lam_l\in\Lam\setminus(\Gamma_0\cup G_0)$ is a removable
singularity for the right hand side of \eqref{cauchytransform:eig}, due to the
continuity of the left hand side.
\end{remark}

\begin{proof}[Proof of Proposition~\ref{prop:eig:aux}]
As mentioned before, the dominant term in Prop.~\ref{proposition:Widomformula}
for $n$ large is obtained by taking $S=S_0:=\{1,2,\ldots,q\}$. Then we find in
the same way as in \cite[Proof of Corollary 5.3]{DK} and \cite[Proof of
Prop.~4.2]{DD} that
\begin{multline}\lim_{n\to\infty} \int_{\cee} \frac{d\mu_{0,n}(x)}{\lam-x}
= \lim_{n\to\infty} \frac{1}{n}\sum_{\lam_i\in\textrm{sp}\
T_n(A)}\frac{1}{\lam-\lam_i} =  \lim_{n\to\infty} \frac{1}{n}\frac{(\det
T_n(A(z)-\lam I_r))'}{\det T_n(A(z)-\lam I_r)}
\\ = \frac{w_{S_0}'(\lam)}{w_{S_0}(\lam)} \label{eigTn:4}  =
-\frac{w_0'(\lam)}{w_{0}(\lam)}+ \sum_{l=1}^L \frac{\mult_{\lam-\lam_l}
f_{-q}}{\lam-\lam_l},
\end{multline}
uniformly on compact subsets of $\cee\setminus(\Gamma_0\cup G_0)$, where the
last equality in \eqref{eigTn:4} follows from \eqref{def:wk} and
\eqref{def:wS}. Finally, Prop.~\ref{prop:cauchytransformcontour} shows that the
right hand side of \eqref{eigTn:4} equals the right hand side
of~\eqref{cauchytransform:eig}.
\end{proof}

\begin{proof}[Proof of Theorem \ref{theorem:eigTn}]
From the convergence of the Cauchy transforms in Prop.~\ref{prop:eig:aux} we
obtain
$$\mu_{0,n}\ \to \ \mu_{0}+\sum_{l=1}^L m_{0}^{(l)}\delta_{\lam_l}$$
in the weak-star sense, i.e., \eqref{weakconvergence} holds for every
continuous function $\phi$ that vanishes at infinity.

Gerschgorin's circle theorem implies that there is a compact set $K$ such that
all the measures $\{\mu_{0,n}\}_{n}$ are supported in $K$. Therefore the
assumption that $\phi$ vanishes at infinity is redundant and we obtain
\eqref{weakconvergence} for all bounded continuous functions.

Finally, we establish the claim that $G_0$ attracts isolated eigenvalues. Let
$\lam_0\in G_0$ and take a sufficiently small disk $U$ around $\lam_0$. Then
from Prop.~\ref{proposition:Widomformula} we find that
\begin{equation}\label{proof:isolated} w_{S_0}(\lam)^{-n-\alpha}\det T_n(A-\lam
I_r) = C_{S_0}(\lam)+O(c^{n}),\qquad \lam\in U,
\end{equation}
for some absolute constant $c$ with $|c|<1$. We claim that $w_{S_0}(\lam)$
tends to a non-zero constant if $\lam\to\lam_0$. This is obvious if
$\lam_0\not\in\Lam$; if $\lam_0=\lam_l\in\Lam$ then it follows by mimicking
\eqref{proof:zeroorder} and noting that $m_0^{(l)}=0$ due to our assumption
that $\lam_0=\lam_l\in G_0\subset\cee\setminus\Gamma_0$. From
\eqref{proof:isolated}, Hurwitz' theorem then implies that for all $n$
sufficiently large, there are precisely $j$ eigenvalues (counting
multiplicities) of $T_n(A)$ inside $U$, with $j$ being the multiplicity of
$\lam_0$ as a zero of $C_{S_0}=C_0$.

Finally let us note that, strictly speaking, the above applications of
Prop.~\ref{proposition:Widomformula} again require that there is a disk $U$
around $\lam_0$ such that for all but finitely many $\lam\in U$ the roots to
$f(z,\lam)=0$ are pairwise different. But this constraint can again be
circumvented by an easy continuity argument.
\end{proof}

\section{Examples}
\label{section:examples}

\subsection{Example 1: a non-degenerate case}
\label{subsection:examplenondeg}

We now illustrate our main results for a small-size example where each of the
hypotheses H1, H2$k$ and H3 holds true. Consider the symbol
\begin{equation}\label{example:symbol2x2} A(z) =
\begin{pmatrix}
b_1 & a_1+c_1z \\
c_2+a_2/z & b_2
\end{pmatrix},
\end{equation}
where we assume for convenience that each of the numbers $a_j,c_j$,
$j\in\{1,2\}$, is non-zero. Then the block Toeplitz matrix $T_n(A)$ has the
tridiagonal form
$$
T_n(A) = \left(\begin{array}{cc|cc|c}
b_1 & a_1 & & & 0 \\
c_2 & b_2 & a_2 & & \\
\hline & c_1 & b_1 & a_1 & \\
& & c_2 & b_2 & \ddots \\
\hline 0 & & & \ddots & \ddots  \\
\end{array}\right)_{2n\times 2n}.
$$
A little calculation shows that
\begin{eqnarray*} f(z,\lambda) & = &
-c_1c_2z+((b_1-\lam)(b_2-\lam)-a_1c_2-a_2c_1)-\frac{a_1a_2}{z}\\
& =: & -\frac{c_1c_2}{z} (z-z_1(\lam))(z-z_2(\lam)),
\end{eqnarray*} where as usual the roots are ordered such that $|z_1(\lam)|\leq
|z_2(\lam)|$. We now have $p=q=1$ and hence there is only one relevant set
$$ \Gamma_0 = \{\lam\in\cee\mid |z_1(\lam)|= |z_2(\lam)|\}.
$$

The coefficients $C_S(\lam)$ in Prop.~\ref{proposition:Widomformula} are
labeled by index sets $S\subset\{1,2\}$ with $|S|=1$; hence $S=\{1\}$ or
$S=\{2\}$. It can be shown that
$$ C_S(\lam) = -\frac{1+\frac{c_1}{a_1}z_i}{c_1c_2(z_i-z_j)},
$$
where we put $S=\{i\}$, $i\in\{1,2\}$, and where $j\in\{1,2\}$ is the index
different from $i$. In particular, we have $C_S(\lam) = 0$ if and only if
$z_i(\lam) = -a_1/c_1$.
From \eqref{example:symbol2x2}, this implies in turn that
$$ \det\begin{pmatrix}
b_1-\lam & 0 \\
* & b_2-\lam
\end{pmatrix} =0.
$$
Hence, we can only have $C_S(\lam) = 0$ if $\lam = b_1$ or $\lam = b_2$. For
these two special $\lam$-values, the second solution to $f(z,\lam)=0$ is
$z_j(\lam)=-a_2/c_2$; therefore we obtain that
\begin{equation}\label{example:symbol2x2:Gamma0} G_0 =
\left\{\begin{array}{cl}
\{b_1,b_2\}, & \quad\textrm{ if $|a_1/c_1|\leq |a_2/c_2|$},\\
\emptyset, & \quad\textrm{ otherwise}.\end{array}\right.
\end{equation}

\begin{figure}[t]
\begin{center}
\includegraphics[scale=0.35,angle=270]{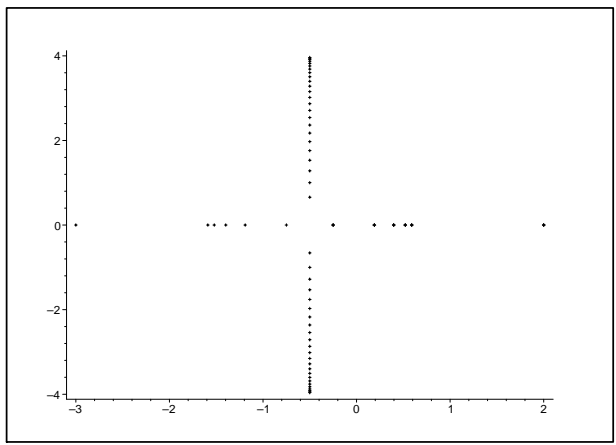}
\end{center}
\caption{Eigenvalues for the symbol \eqref{example:symbol2x2}, with the values
\eqref{example:symbol2x2:values} and $n=30$, computed in Maple with 60 digit
precision. There are $2n = 60$ eigenvalues in total of which $48$ live on the
vertical line segment $[x_4,x_3]\approx[-1/2-3.97i,-1/2+3.97 i]$, $10$ live on
the horizontal line segment $[x_1,x_2]\approx [-1.61,0.61]$, and the final $2$
are outliers lying extremely close to $\lam=-3$ and $\lam=2$, respectively,
cf.~\eqref{example:symbol2x2:Gamma0:bis}.} \label{fig:Widom:Vb3}
\end{figure}
We now turn the above discussion into a numerical example by setting
\begin{equation}\label{example:symbol2x2:values} a_1=1/2,\ a_2=-3,\ b_1=2,\ b_2=-3,\ c_1=4,\ c_2=-3.
\end{equation}
In this case, the discriminant of $f(z,\lam)=0$ equals
$$\lambda^4+2\lambda^3+16\lambda^2+15\lambda-63/4, $$ whose four roots are
$x_1\approx -1.61$, $x_2\approx 0.61$, $x_3\approx -1/2+3.97 i$ and $x_4\approx
-1/2-3.97 i$. These are the \emph{branch points} of $f(z,\lam)=0$. It turns out
that the set $\Gamma_0\subset\cee$ consists of two line segments, one
vertically connecting the branch points $x_3$ and $x_4$ and the other one
horizontally connecting the branch points $x_1$ and $x_2$. The two line
segments intersect at $\lam=-1/2$.

For the values \eqref{example:symbol2x2:values}, the first case in
\eqref{example:symbol2x2:Gamma0} applies and so we have
\begin{equation}\label{example:symbol2x2:Gamma0:bis}G_0 = \{b_1,b_2\} = \{-3,2\}.
\end{equation} Thus for $n$ large,
$T_n(A)$ has an isolated eigenvalue near $\lam=-3$ and near $\lam=2$, both of
multiplicity one. Finally, Theorem~\ref{theorem:main1} implies that the
limiting eigenvalue distribution $\mu_0$ of $T_n(A)$ is precisely the
equilibrium measure of the set $\Gamma_0$.

These considerations are confirmed in Figure~\ref{fig:Widom:Vb3}.

\subsection{Example 2: a degenerate case, I}
\label{subsection:exampledegenerate}

Next we study an example where both H1 and H2 fail. Consider the symbol
\cite[page 321]{Widom1}
\begin{equation}\label{Widom:vb1} A(z) = \begin{pmatrix} 0 & z^{-1}-z
\\ 1+z & z^{-1} + z^2\end{pmatrix}.
\end{equation}
Then one has that \begin{equation}\label{example:pointmass2} f(z,\lam) = \det
(A(z)-\lam I_2) = (1-\lam) z^2+z+(\lam^2-1)+(-1-\lam)z^{-1}.\end{equation}
Hence hypothesis~H1 is violated.

Observe that the following factorization holds,
\begin{equation}\label{example:pointmass3} f(z,\lam) =
\frac{1}{z}((1-\lam) z+1)(z^2-\lam-1).\end{equation} Hence the three roots are
given by $\{z_1(\lam),z_2(\lam),z_3(\lam)\} =
\{\frac{1}{\lam-1},(\lam+1)^{1/2},-(\lam+1)^{1/2}\}$, where the labeling should
be taken according to increasing absolute value. Since two of the three roots
have the same absolute value in the entire complex $\lam$-plane,
hypothesis~H2$k$ is violated as well.

Let us first check the point sources. The set $\Lam=\{\lam_l\}_{l=1}^L$ in
\eqref{def:laml} is such that $L=2$ and $\lam_1 = -1$, $\lam_2 = 1$, and the
relevant data are given by
$$
\left.\begin{array}{c|cccc}
 & k=-1 & k=0 & k=1 & k=2 \\
\hline m_k & 1 & 2 & 3/2 & 1 \\
m_k^{(1)} & 1 & 1/2 & 0 & 0\\
m_k^{(2)} & 0 & 0 & 0 & 1 \\
\til m_k & 0 & 3/2 & 3/2 & 0
\end{array}.\right.
$$
From this we see that the measures $\mu_0$ and $\mu_1$ both have total mass
$\til m_0 = \til m_1= 3/2$.

\begin{figure}[t]
\begin{center}
\includegraphics[scale=0.35,angle=270]{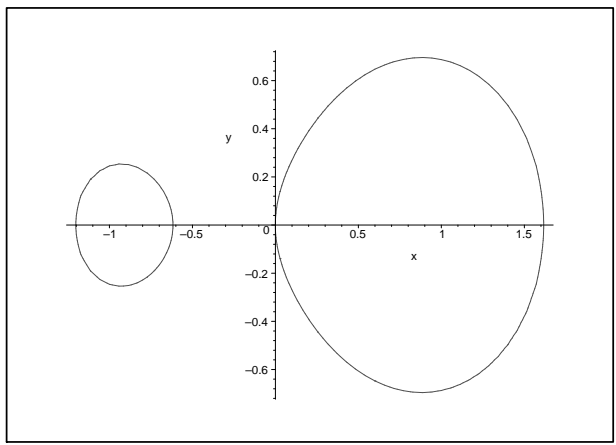}
\end{center}
\caption{Support $\Gamma_0$ of the limiting eigenvalue distribution for the
symbol \eqref{Widom:vb1}. It consists of two closed Jordan curves in the
complex $\lam$-plane and an isolated point at $\lam=-1$.} \label{fig:Widom:Vb1}
\end{figure}

\begin{figure}[t]
\begin{center}
\includegraphics[scale=0.35,angle=270]{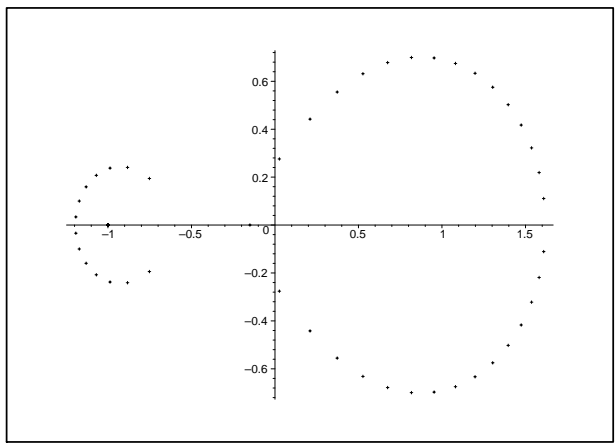}
\end{center}
\caption{Eigenvalues for $n=30$ for the symbol \eqref{Widom:vb1}, computed in
Maple with 60 digit precision. There are $2n = 60$ eigenvalues in total of
which $15$ coalesce at $\lam=-1$. Note that the eigenvalues closely approximate
the curve in Figure \ref{fig:Widom:Vb1}.} \label{fig:Widom:Vb2}
\end{figure}

Taking into account Section~\ref{subsection:Gammak2D},
cf.~\eqref{def:Gammak:mod}, the curves $\Gamma_0$ and $\Gamma_1$ are defined as
$$ \Gamma_1 = \{\lam\in\cee\mid |\lam+1|^{1/2}|\lam-1|=1\},\qquad \Gamma_0 = \Gamma_1\cup\{-1\}.
$$ The curve $\wtil\Gamma_0=\wtil\Gamma_1$ is plotted in Figure
\ref{fig:Widom:Vb1}. On this curve, the measures $\mu_0$ and $\mu_1$ are
defined according to the density \eqref{def:measuresk}.

The failure of H2 implies, as mentioned before, that the definitions of H3 and
$G_0$ need to be modified. Let us do this with an ad-hoc calculation. The
coefficients $C_S(\lam)$ in Prop.~\ref{proposition:Widomformula} are labeled by
index sets $S\subset\{1,2,3\}$ with $|S|=1$; it can be shown that
$$ C_S(\lam) = \left\{\begin{array}{ll}
-\frac{\lam^2}{(\lam+1)(1-(\lam+1)(\lam-1)^2)},&
\textrm{ if $S$ labels the root }z(\lam)=\frac{1}{\lam-1},\\
\frac{1\pm (\lam+1)^{1/2}}{2(\lam+1)(1\pm(1-\lam)(\lam+1)^{1/2})},&\textrm{ if
$S$ labels the root }z(\lam)=\pm(\lam+1)^{1/2}.\end{array}\right.
$$
After some simplifications, \eqref{Widom:1} then reduces to \begin{equation*}
\det T_n(A(z)-\lam I_2) =
\frac{\lam \left(\lam^2-1\right)^{n+1}}{(\lam+1)(\lam^2-\lam-1)}\\
-\left\{\begin{array}{ll}\frac{(\lam+1)^{n/2}}{\lam^2-\lam-1},&\textrm{ if $n$
is even,}\\ \frac{\lam(\lam+1)^{(n-1)/2}}{\lam^2-\lam-1},&\textrm{ if $n$ is
odd}.\end{array}\right.
\end{equation*}
From this, it is easy to see that Theorem~\ref{theorem:eigTn} can indeed be
applied. Thus the limiting eigenvalue distribution of $T_n(A)$ consists of the
absolutely continuous part $\mu_0$ on $\wtil\Gamma_0$, and a point mass of mass
$1/2$ at $\lam_1=-1$. Moreover, $T_n(A)$ does not have isolated eigenvalues for
$n$ large, neither for $n$ even nor for $n$ odd. This reproduces the result in
\cite[page 321]{Widom1}. The comparison with the eigenvalues of $T_n(A)$ for
$n=30$ is shown in Figure \ref{fig:Widom:Vb2}.

Finally, the energy functional \eqref{energyfunctional} reduces to
\begin{equation}\label{energyfunctional:ex2}
I(\nu_0)+I(\nu_1) - I(\nu_0,\nu_1) -\frac 12 \int\log\frac{1}{|\lam+1|}\
d\nu_1(\lam)- \int\log\frac{1}{|\lam-1|}\ d\nu_1(\lam).
\end{equation}
Theorem~\ref{theorem:main1:mod} implies that $(\mu_0,\mu_1)$ is the minimizer
of this functional over all pairs of measures $(\nu_0,\nu_1)$ supported on
$\wtil\Gamma_0=\wtil\Gamma_1$, with total masses $\tilde m_0=\tilde m_1=3/2$.
The last two terms in \eqref{energyfunctional:ex2} can be interpreted as an
attraction of $\mu_1$ towards the points $\lam=-1$ and $\lam=1$.

\subsection{Example 3: a degenerate case, II}
\label{subsection:exampledegenerateII}


We discuss a variant of the previous example. Consider the symbol
\begin{equation}\label{similarity:exact}
A(z)=\begin{pmatrix} z^2-1 & z^{-1}-z \\ 0 & z^{-1}+1
\end{pmatrix}.\end{equation} The algebraic equation $f(z,\lam)=0$
is again given by \eqref{example:pointmass3}. The triangularity of $A(z)$
implies the following factorization for the finite $n$ determinants:
\begin{equation*}\label{similarity:finite:n} \det T_n (A(z)-\lam I_2) =
\det T_n (z^2-1-\lam)\ \det T_n\left(z^{-1}+1-\lam\right)=
(-1-\lam)^n(1-\lam)^n.
\end{equation*}
So the limiting eigenvalue distribution of $T_n(A)$ has a pure point spectrum
with point masses at $\lam=-1$ and $\lam=1$. In particular, it is \emph{not}
related to the measure $\mu_0$ on the set $\wtil\Gamma_0$ in
Fig.~\ref{fig:Widom:Vb1}. Thus Theorem~\ref{theorem:eigTn} breaks down in this
case. The reason is that several coefficients in Widom's formula are
identically zero, and so H3 (actually a modification thereof since H2 fails) is
not valid.

It is straightforward to generalize the above idea: Whenever the symbol
$A(z)-\lam I_r$ is block upper triangular, or can be reduced into block upper
triangular form by means of suitable row and column transformations, then $\det
T_n(A(z)-\lam I_r)$ factorizes into two smaller-size block Toeplitz
determinants. For such symbols $A(z)$, the hypotheses H1 and H2$k$ typically
hold true while H3 and Theorem~\ref{theorem:eigTn} both fail.

Finally, one may argue that the above counterexamples to
Theorem~\ref{theorem:eigTn} are harmless, in the sense that in each case the
eigenvalue problem for $T_n(A)$ can be reduced into two smaller-size eigenvalue
problems. One may wish to construct more interesting examples for which
Theorem~\ref{theorem:eigTn} fails. One way to construct such examples is from
the symbol
\begin{equation}
A(z) = \begin{pmatrix} a(z) & 0 & a_{1,3}(z) \\
0 & a(z) & a_{2,3}(z) \\
a_{3,1}(z) & a_{3,2}(z) & a_{3,3}(z)
\end{pmatrix},
\end{equation}
where $a(z)$, $a_{i,3}(z)$ and $a_{3,i}(z)$, $i=1,2,3$,  are given Laurent
polynomials in $z$. By suitably fine-tuning these Laurent polynomials, and
especially the exponents of their highest and lowest degree terms in $z$, one
may construct symbols $A(z)$ for which H1 and H2$k$ hold true, H3 and
Theorem~\ref{theorem:eigTn} both fail, and for which no reduction to block
upper triangular form is possible. We leave the details to the interested
reader.

\section{Concluding remarks}
\label{section:conclusion}

\indent\textbf{1. Generalizations.} The main Theorem~\ref{theorem:eigTn} was
stated under the following condition: either $\cee\setminus\Gamma_0$ is
connected and $\Gamma_0$ does not have any interior points; or $T_n(A)$ is a
Hessenberg matrix. It is an open problem to generalize this theorem to other
classes of banded block Toeplitz matrices.\smallskip

\textbf{2. Applications.} We expect that our main
Theorems~\ref{theorem:main1}/\ref{theorem:main1:mod} and \ref{theorem:eigTn}
may be used to obtain some results in the theory of multiple and matrix
orthogonal polynomials on the real line. In fact, recently several papers
appeared \cite{DGK,Roman,RZ} which apply the results of Duits and Kuijlaars
\cite{DK} on scalar banded Toeplitz matrices, to the context of multiple
orthogonal polynomials. The recurrence relations of these polynomials lead to a
banded Hessenberg matrix. Typically this matrix is not \emph{exactly} Toeplitz
but only \emph{asymptotically}. More generally, the orthogonality weights may
be varying with $n$, which leads to so-called \emph{locally} Toeplitz matrices
\cite{DGK,Roman,RZ}. We anticipate that more applications of this type may
arise in the future, possibly leading to block (rather than scalar) Toeplitz
matrices. A first application of this kind is given in \cite{BDK}. Finally, we
also anticipate that our results could be used in the context of matrix
orthogonal polynomials on the real line, see e.g.~\cite{Dette,DLS}.

\section*{Acknowledgment}

I would like to thank Martin Bender, Maurice Duits and Arno Kuijlaars for
useful discussions.



\begin{thebibliography}{10}
\bibitem{BaxterSchmidt}
    G. Baxter and P. Schmidt,
    Determinants of a certain class of non-Hermitian Toeplitz matrices,
    Math. Scand. 9 (1961), 122--128.
\bibitem{BDK}
    M. Bender, S. Delvaux and A.B.J. Kuijlaars,
    Multiple Meixner-Pollaczek polynomials and the six-vertex
    model,
    J. Approx. Theory 163 (2011), 1606--1637.
\bibitem{BG}
    A. B\"ottcher and S.M. Grudsky,
    Spectral Properties of Banded Toeplitz Matrices,
    SIAM, Philadelphia, PA, 2005.
\bibitem{BS2}
    A. B\"ottcher and B. Silbermann,
    Introduction to Large Truncated Toeplitz Matrices,
    Universitext, Springer-Verlag, New York, 1998.
\bibitem{BS}
    A. B\"ottcher and B. Silbermann,
    Invertibility and Asymptotics of Toeplitz Matrices,
    Akademie-Verlag, Berlin, 1983.
\bibitem{DD}
    S. Delvaux and M. Duits,
    An equilibrium problem for the limiting eigenvalue distribution of rational Toeplitz matrices,
    SIAM. J. Matrix Anal. Appl. 31 (2010), 1894--1914.
\bibitem{Dette}
    H. Dette and B. Reuther,
    Random Block Matrices and Matrix Orthogonal Polynomials,
    J. Theor. Probab. (2008), DOI 10.1007/s10959-008-0189-z.
\bibitem{DGK}
    M. Duits, D. Geudens and A.B.J. Kuijlaars,
    A vector equilibrium problem for the two-matrix model in the quartic/quadratic case,
    Nonlinearity 24 (2011), 951--993. 
\bibitem{DK}
    M. Duits and A.B.J. Kuijlaars,
    An equilibrium problem for the limiting eigenvalue distribution of banded Toeplitz matrices,
    SIAM J. Matrix Anal. Appl. 30 (2008), 173--196.
\bibitem{DLS}
    A. Dur\'an, P. L\'opez-Rodriguez and E.B. Saff,
    Zero asymptotic behaviour for orthogonal matrix polynomials,
    J. Anal. Math. 78 (1999), 37--60.
\bibitem{Goss}
    D. Goss, Basic structures of function field arithmetic, Ergebnisse
    der Mathematik und ihrer Grenzgebiete (3), [Results in Mathematics and Related Areas (3)],
    35, Springer-Verlag, Berlin, 1996.
\bibitem{Hirschman}
    I.I. Hirschman, Jr.,
    The spectra of certain Toeplitz matrices,
    Illinois J. Math. 11 (1967), 145--149.
\bibitem{Roman}
    A.B.J. Kuijlaars and P. Rom\'an,
    Recurrence relations and vector equilibrium problems arising from a model of non-intersecting squared Bessel paths,
    J. Approx. Theory 162 (2010), 2048--2077.
\bibitem{NS}
    E.M. Nikishin and V.N. Sorokin,
    Rational Approximations and Orthogonality,
    Amer. Math. Soc., Providence, RI, 1991.
\bibitem{Rans}
    T. Ransford,
    Potential Theory in the Complex Plane,
    London Math. Soc. Stud. Texts 28,
    Cambridge University Press, Cambridge, UK, 1995.
\bibitem{SaffTotik}
    E.B. Saff and V. Totik,
    Logarithmic Potentials with External Field,
    Springer-Verlag, Berlin, 1997.
\bibitem{SchmidtSpitzer}
    P. Schmidt and F. Spitzer,
    The Toeplitz matrices of an arbitrary Laurent polynomial,
    Math. Scand. 8 (1960), 15--38.
\bibitem{Simeonov}
    P. Simeonov,
    A weighted energy problem for a class of admissible weights,
    Houston J. Math. 31 (2005), 1245--1260.
\bibitem{Ullman}
    J.L. Ullman,
    A problem of Schmidt and Spitzer,
    Bull. Amer. Math. Soc. 73 (1967), 883--885.
\bibitem{Widom0}
    H. Widom,
    On the eigenvalues of certain Hermitian operators,
    Trans. Amer. Math. Soc. 88 (1958), 491--522.
\bibitem{Widom1}
    H. Widom,
    Asymptotic behavior of block Toeplitz matrices and determinants,
    Advances in Math. 13 (1974), 284--322.
\bibitem{Widom2}
    H. Widom,
    Asymptotic behavior of block Toeplitz matrices and determinants, II,
    Advances in Math. 21 (1976), 1--29.
\bibitem{RZ}
    L. Zhang and P. Rom\'an,
    Asymptotic zero distribution of multiple orthogonal polynomials associated with Macdonald
    functions, J. Approx. Theory 163 (2011), 143--162. 
\end{thebibliography}
\end{document}